\documentclass[a4paper,11pt,leqno]{article}

\usepackage[latin1]{inputenc}
\usepackage[T1]{fontenc} 
\usepackage[english]{babel}
\usepackage{verbatim}
\usepackage{calligra}
\usepackage{enumerate}
\usepackage{dsfont}
\usepackage{amsfonts}
\usepackage{amsmath}
\usepackage{amsthm}
\usepackage{amssymb}
\usepackage{fancyhdr}
\usepackage{mathtools}
\usepackage{mathrsfs}
\usepackage{color}
\usepackage{graphicx}
\usepackage{bm}

\theoremstyle{definition}
\newtheorem{defin}{Definition}[section]

\newtheorem{rem}[defin]{Remark}

\theoremstyle{plane}
\newtheorem{thm}[defin]{Theorem}

\newtheorem{lemma}[defin]{Lemma}

\newcommand{\tsl}{\textsl}

\newcommand{\mbb}{\mathbb}

\newcommand{\mc}{\mathcal}
\newcommand{\mf}{\mathfrak}
\newcommand{\mds}{\mathds}
\newcommand{\veps}{\varepsilon}

\newcommand{\wtilde}{\widetilde}
\newcommand{\vphi}{\varphi}
\newcommand{\oline}{\overline}

\newcommand{\ra}{\rightarrow}

\newcommand{\g}{\gamma}
\renewcommand{\k}{\kappa}
\newcommand{\s}{\sigma}
\renewcommand{\o}{\omega}

\newcommand{\lan}{\langle}
\newcommand{\ran}{\rangle}

\newcommand{\dd}{\textnormal{d}}
\newcommand{\dx}{ \, {\rm d} x}

\newcommand{\R}{\mathbb{R}}
\newcommand{\Q}{\mathbb{Q}}

\newcommand{\N}{\mathbb{N}}
\newcommand{\Z}{\mathbb{Z}}
\newcommand{\T}{\mathbb{T}}
\newcommand{\D}{\mathbb{D}}
\renewcommand{\P}{\mathbb{P}}

\renewcommand{\div}{{\rm div}\,}
\newcommand{\divh}{{\rm div}_h}
\newcommand{\curlh}{{\rm curl}_h}

\newcommand{\curl}{{\rm curl}\,}

\newcommand{\Id}{{\rm Id}\,}
\newcommand{\Supp}{{\rm Supp}\,}

\newcommand{\ess}{{\rm ess}}
\newcommand{\res}{{\rm res}}
\newcommand{\loc}{{\rm loc}}

\allowdisplaybreaks

\def\d{\partial}
\def\div{{\rm div}\,}

\title{\textsc{\Large{\textbf{Incompressible limits at large Mach number for \\ 
a reduced compressible MHD system}}}}

\author{\normalsize\textsl{Francesco Fanelli}$\,^{1a,b,c}\,$, $\qquad$
\textsl{Young-Sam Kwon}$\,^{2}\,$, $\qquad$ \textsl{Aneta Wr\'oblewska-Kami\'nska}$\,^{3}$ \vspace{.5cm} \\
\footnotesize{$\,^{1a}\;$ \textsc{BCAM -- Basque Center for Applied Mathematics}} \\ 
\footnotesize{Alameda de Mazarredo 14, E-48009 Bilbao, Basque Country, SPAIN}  \vspace{0.2cm} \\
\footnotesize{$\,^{1b}\;$ \textsc{Ikerbasque -- Basque Foundation for Science}} \\
\footnotesize{Plaza Euskadi 5, E-48009 Bilbao, Basque Country, SPAIN} \vspace{.2cm} \\
\footnotesize{$\,^{1c}\;$ \textsc{Univ. Lyon, Universit\'e Claude Bernard Lyon 1}, CNRS UMR 5208, \textit{Institut Camille Jordan},} \\
\footnotesize{43 blvd. du 11 novembre 1918, F-69622 Villeurbanne cedex, FRANCE} \vspace{.2cm} \\
\footnotesize{$\,^{2}\;$ \textsc{Dong-A University}, \textit{Departement of Mathematics},} \\ 
{\footnotesize Nakdongdaero 550-37 Sahagu, 49315 Busan, REPUBLIC OF KOREA} \vspace{.2cm} \\
\footnotesize{$\,^{3}\;$ \textsc{Institute of Mathematics of Polish Academy of Sciences},} \\
{\footnotesize ul.\'Sniadeckich 8, 00-656 Warszawa, POLAND} \vspace{.2cm} \\
\footnotesize{Email addresses: $\,^{1}\;$\ttfamily{ffanelli@bcamath.org}}, $\;$
\footnotesize{$\,^{2}\;$\ttfamily{ykwon@dau.ac.kr}, $\;$ $\,^{3}\;$\ttfamily{a.wroblewska@impan.pl}}
\vspace{.2cm}
}

\vspace{.3cm}
{\footnotesize \date\today}

\begin{document}

\maketitle

\subsubsection*{Abstract}
{\footnotesize

This paper studies a singular limit problem for a reduced model for compressible non-resistive MHD which was first introduced in \cite{Li-Sun_JDE, Li-Sun}
in a two-dimensional setting. This system can also be related to a certain class of two-fluid models.
By a suitable rescaling of the magnetic pressure in terms of some
parameter $\veps>0$, by letting $\veps\to 0$ we perform the incompressible limit while keeping the Mach number of order $O(1)$.

The study is conducted in the framework of global in time finite energy weak solutions and for ill-prepared initial data.
We also consider a similar problem in presence of a strong Coriolis term.
The key ingredient of the proof, based on a compensated compactness argument, is the use of the transport equation (well-known in the context
of two-fluid models) underlying the dynamics. Thanks to it, and differently from previous studies about the incompressible limit,
we are able to identify the asymptotics of the terms of order $O(\veps)$ and to characterise their dynamics; such an information
is in fact crucial to obtain a closed system in the limit.

}

\paragraph*{2020 Mathematics Subject Classification:}{\small 35Q35 
(primary);
35B25, 
35B40, 
76W05 
(secondary).}

\paragraph*{Keywords:}{\small reduced compressible MHD system; two-fluid system; incompressible limit; large Mach number; ill-prepared initial data;
strong Coriolis force.}

\section{Introduction} \label{s:intro}

Owing to its relevance in several physically motivated contexts, like for instance the description of large scale flows, the rigorous study of the
incompressible limit for weakly compressible fluids has attracted a lot of interest from the mathematical community in the last decades.
This has been performed in several frameworks (weak solutions \tsl{vs} strong solutions), for several type of domains (the whole space, the periodic setting,
a general bounded domain or an exterior domain), under various assumptions on the fluid (viscous \tsl{vs} inviscid cases, isentropic \tsl{vs} heat-conducting
fluids) and on the initial data (which may be either well-prepared or ill-prepared).

The classical strategy to perform the incompressible limit relies on the so-called \emph{low Mach number} regime
(see \tsl{e.g.} \cite{Kl-Maj, L-Masm, Hag-L, Des-G, Des-G-L-M}). The Mach number ${\rm Ma}$ is an
adimensional physical parameter defined as the ratio between the typical speed of the fluid and the speed of sound in the medium. When computing an
adimensionalisation of the equations of motion, the square of the Mach number appears as a prefactor in front of the pressure gradient: more precisely,
in the momentum equation the following term comes into play,
\[
 \frac{1}{{\rm Ma}^2}\,\nabla P(\rho)\,,
\]
where $\rho$ is the fluid density and $P$ is the fluid pressure. Thus, when letting ${\rm Ma}\to0$, the density converges (at least, formally) towards a constant state
(in fact, one can discard the time-dependence from the mass conservation equation),
which implies, thanks to the mass conservation equation, that the target velocity $u$ must satisfy the constraint $\div u=0$.
We refer also to \cite{Met-Sch, Al, F-N_2007, F-N_2009} for more advances in the case of non-isentropic flows, as well as
to book \cite{F-N} for a complete overview of the subject and additional references.
We refrain from mentioning here related results, obtained by following a similar strategy, devoted to the rigorous justification of the anelastic approximation
(see \cite{F-Z} and references therein).

More recently, Danchin and Mucha \cite{D-M, D-M_Tun} proposed another method to perform the incompressible limit (see also \cite{F_PhysD} for a related study).
Their approach relies on the penalisation of the \emph{bulk viscosity coefficient}, namely of the term
\[
 \lambda\,\nabla\div u
\]
appearing in the momentum equation. Sending $\lambda\to+\infty$ directly implies, from energy estimates, that the limit velocity field must satisfy the divergence-free
constraint as before, but this time with no more need of imposing the Mach number to be small.

\medbreak
The goal of the present paper is to propose another method to perform the incompressible limit, in the context of a certain class of \emph{two-fluid models}.
To fix ideas, we will consider (see Section \ref{ss:system} below) a reduced MHD system first introduced in \cite{Li-Sun_JDE, Li-Sun}
in two space dimensions; while keeping the same mathematical structure of the equations, for the sake of generality we will consider the three-dimensional
version of that system.

In those equations, we will consider the Mach number and the bulk viscosity coefficient of order $O(1)$,
meaning that these parameters will not be rescaled and they will be kept fixed in our study.
The idea for performing the incompressible limit is instead to penalise the gradient of the magnetic pressure
(in terms of two-fluid systems, we penalise only one pressure gradient while keeping the other at constant order of magnitude), and to use
a transport structure  underlying the equations to gain smallness on the perturbations of the density around some given constant state.

To explain better this idea, let us first introduce the system of equations, together with the right scaling, we will be considering throughout this work. Then,
in Subsection \ref{ss:over} we will give an overview of the results and of the main ideas of the proof; we will conclude the discussion
with a short comparison of our study with related results available in the literature.

\subsection{The system of equations} \label{ss:system}
Let $\veps>0$ and $\k\in\{0,1\}$ be two parameters, with the idea that $\k$ will allow us to turn on and off the effects of the Coriolis force, but, once chosen,
it will be kept fixed while letting $\veps$ going to $0$.
The system of equations we are going to consider in this paper is the following one:
\begin{equation} \label{eq:reduced-MHD}
\left\{\begin{array}{l}
        \d_t\rho\,+\,\div\big(\rho\,u\big)\,=\,0 \\[1ex]
       \d_t\big(\rho\,u\big)\,+\,\div\big(\rho\,u\otimes u\big)\,+\,\nabla P(\rho)\,+\,\dfrac{1}{2\,\veps^{2}}\,\nabla b^2\,+\,
       \dfrac{\k}{\veps}\,e_3\times\rho\,u\,-\,\div\mbb S(\D u)\,=\,0 \\[2ex]
       \d_tb\,+\,\div\big(b\,u\big)\,=\,0\,.
       \end{array}
\right.
\end{equation}
Here, the unknowns are the two scalar functions $\rho\geq0$ and $b\in\R$, and the vector-field $u\in\R^3$.
From a pure modelling perspective (see more details below), the function $\rho$ represents the density of the fluid (whence the requirement of positivity) and
$u$ its velocity field, whereas $b$ represents (one component of) the magnetic field, the fluid being assumed to be conductive.
The function $P=P(\rho)$ represents instead the fluid pressure; here, we assume that
\begin{align} \label{hyp:pressure}
&P\in C^1\big(\R_+\big)\cap C^2\big(\,]0,+\infty[\,\big)\,,\qquad P(0)=0\,, \\
\nonumber
&\qquad  P'(\rho)>0\qquad \forall\,\rho>0\,,\qquad\qquad\qquad \underline{P}\,\rho^{\g-1}\,\leq\,P'(\rho)\,\leq\,\overline{P}\,\rho^{\g-1}\qquad \forall \rho>1\,,
\end{align}
for some $\g>1$ and for two suitable constants $0<\underline{P}\leq \overline{P}$.

The fluid is supposed to be Newtonian, so that the viscous stress tensor $\mbb S$ takes the form
\begin{equation} \label{hyp:S}
\mbb S(\D u)\,:=\,\mu\left(D u\,+\,\nabla u\,-\,\frac{2}{3}\,\div u\,\Id\right)\,+\,\lambda\,\div u\,\Id\,,
\end{equation}
where $\D u$ is symmetric part of the Jacobian matrix $Du$ of the velocity field; namely,
$\D u = (Du + \nabla u)/2$, where we have set $\nabla u\,=\,^tDu$. In this paper, we suppose $\mu$ and $\lambda$ to be constants such that
\[
\mu\,>\,0\qquad\qquad \mbox{ and }\qquad\qquad \frac{4}{3}\,\mu\,+\,\lambda\,>\,0\,.
\]
Observe that $\div \mbb S(Du)\,=\,\mu\,\Delta u\,+\,\left(\frac13\mu+\lambda\right)\,\div u$.

Finally, the term $e_3\times \rho\, u$ allows us to take into account (when $\k=1$) the possible action of the Coriolis force; this is
relevant at geophysical scales (see \tsl{e.g.} \cite{C-R, Ped}), where the fast rotation of the Earth (or of any other celestial body) represents one
of the major factors which affect the fluid motion.
Here, $e_3=(0,0,1)$ denotes the unit vector directed along the vertical component and, for $V\in\R^3$, we have set
$e_3\times V = (-V_2, V_1, 0)$ to be the usual external product in $\R^3$.

\paragraph{The domain.}
System \eqref{eq:reduced-MHD} is set for $(t,x)\,\in\,\R_+\times\Omega$, where,
according to the available existence theory, $\Omega\subset\R^3$ must be a bounded smooth domain.
In order to avoid complications due to boundary effects, for simplicity in this paper we will assume that
\begin{align*}
&\mbox{when }\quad \k=0\,,\qquad\qquad 
\Omega\,=\,\T^3\,, \\[1ex] 
&\mbox{ and, \ when }\quad \k=1\,, \qquad\qquad \Omega\,=\,\T^2\times\,]0,1[\,\,.
\end{align*}
In the latter case, we supplement the system with complete-slip boundary conditions:
\begin{equation} \label{bc:compl-slip}
 \big(u\cdot n\big)_{|\d\Omega}\,=\,0\qquad \mbox{ and }\qquad \Big(\big(\D u\cdot n\big)\times n\Big)_{|\d\Omega}\,=\,0\,,
\end{equation}
where $n\,=\,\pm\,e_3$ is the external unit normal vector to the boundary $\d\Omega$ of the domain.
Observe that, despite the physical relevance of a periodic domain in presence of fast rotation is certainly questionable,
our study is able to capture its main effects, and to show new interesting features linked to it at the level of the limit dynamics.

\begin{rem} \label{r:bc-period}
As is well-known (see \tsl{e.g.} \cite{Ebin}), problem \eqref{eq:reduced-MHD}, set in $\Omega = \T^2\times\,]0,1[\,$ and supplemented with
complete-slip boundary conditions \eqref{bc:compl-slip}, can be recasted as a periodic problem also with respect to
the third variable (see also Subsection \ref{ss:k=1_convect} for more details about this).

Therefore, when reviewing the weak solutions theory (see Subsection \ref{ss:weak-sol} below),
we can safely restrict our discussion to weak solutions defined on a period box.
\end{rem}

Before moving on, let us point out that, although at first sight the boundedness of $\Omega$ does not look essential for the asymptotic study,
we will use this property in some crucial estimates. Thus, the generalisation of our results to unbounded domains (for instance,
$\Omega=\R^3$ when $\k=0$, or $\Omega=\R^2\times\,]0,1[\,$ for $\k=1$) remains open.


\paragraph{Origin of the model.}
System \eqref{eq:reduced-MHD} represents the three-dimensional version of the reduced planar MHD system introduced in \cite{Li-Sun_JDE, Li-Sun}.
This is why we have adopted the terminology of those papers when interpreting the various unknowns appearing in the equations.
For the sake of completeness, we mention that the reduced system was formally derived from a general compressible non-resistive magnetohydrodynamics system in $3$-D,
by looking for solutions for which all the unknowns depend only on the horizontal variables $x_h=(x_1,x_2)$, the velocity field
$U=(u_1,u_2,0)$ has zero third component 
and the magnetic field $B=(0,0,b)$ is simply vertical.

However, we remark that system \eqref{eq:reduced-MHD} can be seen also as a special type of two-fluid model, and in fact the results
of \cite{Li-Sun_JDE, Li-Sun} exploit this analogy in their investigations.
The analysis of two-fluid models are nowadays central in many mathematical works
(see \tsl{e.g.} \cite{V-W-Y, N-P, Li-Z, Pok-WK-Z, B-B-L, Zod}, just to mention a few).
This represents a strong motivation for us to keep a general perspective here and consider equations \eqref{eq:reduced-MHD} in three space 
dimensions.

\paragraph{The scaling.}
Before moving on, let us comment a bit on the scaling appearing in equations \eqref{eq:reduced-MHD}.
We see that the pressure term $\nabla P(\rho)$ is not penalised, and the value of the bulk viscosity coefficient $\lambda$ is kept fixed.
Instead, a large prefactor $\veps^{-2}$ appears in front of the magnetic pressure term. Taking the limit $\veps\to0$ activates the following mechanism:
the function $b=b_\veps$ must converge to a constant value, precisely as it happens in the (classical) low Mach number regime for compressible fluid flows;
in turn, from the continuity equation for $b$, this yields that the target velocity field must be divergence-free.
So, as anticipated above, our scaling does allow to perform an incompressible limit, while keeping large vaules of the Mach number (here set equal to $1$)
and small values  the bulk viscosity coefficient $\lambda$.

Of course, the above mentioned heuristic argument is not enough to compute the limit of the equations: more details will be given in the next subsection.
Here, we want to spend a few more words on the relevance of the imposed scaling.
As a matter of fact, besides the mathematical interest of this new approach to the incompressible limit, the scaling appearing in \eqref{eq:reduced-MHD},
yielding the large magnetic pressure term $\frac{1}{\veps^2}\nabla b^2$, is physically pertinent in the context
of conducting fluids and their description through the MHD equations. 
Indeed, according to the discussion in Subsection 2.7 of \cite{Bellan}, the only non-trivial situations in which the MHD system finds its full relevance
are either the case in which the magnetic force and the hydrodynamic force are comparable, or the one in which the former is predominant over the latter.
From this viewpoint, our limit $\veps\to0$ can be seen as an investigation of what happens in the second setting, in the limit in which the magnetic forces
are much stronger than the hydrodynamic ones.

\subsection{Overview of the results and strategy of the proof} \label{ss:over}

As mentioned before, given a family of solutions $\big(\rho_\veps,u_\veps,b_\veps\big)_{\veps}$ to system \eqref{eq:reduced-MHD},
the goal of this paper is to prove their convergence to some target state $\big(\rho,u,b\big)$, which we want to characterise
together with its dynamics. We will perform this study in the class of finite energy weak solutions, as defined in \cite{Li-Sun_JDE}, and for general
ill-prepared initial data. We will consider both the case of absence $\k=0$ and presence $\k=1$ of a strong Coriolis force, obtaining qualitatively different target dynamics.
When $\k=0$, we prove convergence to a $3$-D incompressible Navier-Stokes system for the target velocity field $u$;
when $\k=1$, instead, we prove convergence towards a modified
quasi-geostrophic type equation, coupled with a scalar equation for a new quantity $\alpha$ (more details will follow). The precise statements, which represent
the main contributions of this work, are given respectively in Theorems \ref{th:limit} and \ref{th:lim_Coriols}.

We have already explained that the scaling considered here is related to an incompressible limit, for which we expect to get $b\equiv\oline b$ constant and
$\div u = 0$.
Despite the heuristic argument is very simple and seems to impose no special constraints on the target density (which, in principle, could even be
a non-constant function), the coupling of the equations is actually very strong and introduces severe restrictions when rigorously
computing the asymptotics $\veps\to0$.
Roughly speaking, the propagation of acoustic 
waves, generated by the fast time oscillations due to the singular terms in the equations, is governed by the large gradients of the magnetic functions $b_\veps$
and the divergence of the vector fields $W_\veps\,:=\,b_\veps\,u_\veps$.
However, the waves affect the dynamics of the momenta $V_\veps\,:=\,\rho_\veps\,u_\veps$
and, at the same time, it is precisely this quantity which enters into play in the convergence of the convective term $\rho_\veps u_\veps\otimes u_\veps$,
which is the most non-linear quantity appearing in \eqref{eq:reduced-MHD}.
When $\k=1$, the situation is even more complex, because (as already remarked in \tsl{e.g.} \cite{F-G-N, F-G-GV-N, F_PhysD})
the additional penalisation of the Coriolis term strictly couples with the mass conservation balance in the limit $\veps\to0$.
In the end, a crucial point of the proof is to make
the functions $b_\veps$ communicate with the densities $\rho_\veps$, or in other words to compare the velocity fields $W_\veps:=b_\veps\,u_\veps$
with the momenta $V_\veps:=\rho_\veps\,u_\veps$.

In order to solve this issue, 
in the present paper we restrict our attention to density functions which are small perturbations, of size
$O(\veps)$, of a \emph{constant} reference state $\oline\rho$:
\begin{equation} \label{ass:rho}
\rho_{0,\veps}\,=\,\oline\rho\,+\,O(\veps)\,.
\end{equation}
It is not clear to us that this assumption is really strictly necessary to perform the asymptotic analysis for $\veps\to0$,
but it comes into play in a crucial way at least twice, in our study.
Notice that, by resorting to the transport equation satisfied by the quantities $b_\veps/\rho_\veps$, a property
which is well-known in the context of two-fluid systems \cite{V-W-Y}, and thanks to the smallness $b_\veps = \oline b + O(\veps)$ (which follows from the energy
balance), we are able to propagate the property $\rho_\veps = \oline\rho + O(\veps)$ also for later times. This is the key information
to unravel the intricate structure of the wave system and get the sought link between $b_\veps$ and $\rho_\veps$, thus between $W_\veps$ and $V_\veps$.

Actually, this is not completely the end of the story. Resorting to a compensated compactness argument for studying the convergence
of the convective term  $\rho_\veps u_\veps\otimes u_\veps$, a last difficulty appears, encoded by the appearance of a new non-linear term.
Precisely, if we now write the more accurate expressions $b_\veps = \oline b + \veps \beta_\veps$ and $\rho_\veps = \oline\rho + \veps r_\veps$, 
one can see that the new non-linear term arising in the computations depends on the product $r_\veps\,\nabla \beta_\veps$. Thus, for passing to the limit in this term,
some strong convergence property is needed on one of the two quantities. We remark that this issue requires to go one step further with respect to the classical
incompressible limit \tsl{via} the low Mach number regime, for which the limit of the density perturbation functions (the equivalent of the $r_\veps$'s here)
is not, and in general \emph{cannot}, be determined. 
Once again, what makes the difference in our study is exactly the above assumption \eqref{ass:rho}
together with the fact that this property propagates also for later times. Indeed, using the structure (reminiscent of the previously
mentioned transport property for the quotients $b_\veps/\rho_\veps$) of the continuity equations and of the solutions $b_\veps$ and $\rho_\veps$, one can
prove that the quantities $\alpha_\veps\,:=\,\oline\rho\,\beta_\veps\,-\,\oline b\,r_\veps$ are compact (in suitable topologies), so they strongly converge
to some scalar quantity $\alpha$ (precisely the function $\alpha$ named at the beginning of this subsection).
This is the crucial information which allows us to pass to the limit in the new non-linear term $r_\veps\,\nabla \beta_\veps$ and finally prove
the convergence results when $\veps\to0$.

\medbreak
To conclude this part, let us mention that, to the best of our knowledge, the scaling proposed in \eqref{eq:reduced-MHD} and the mechanism (depicted above) leading
to the incompressible limit are \emph{new}. In addition, we belive that, although presented only for the magnetic pressure term $\nabla b^2$,
this method could certainly be implemented in presence of different power laws, more pertinent in the context of two-fluid equations.

For systems keeping the same structure of our system \eqref{eq:reduced-MHD}, we are aware of only a few studies \cite{F-K-N-Z, Li-Mu} on the incompressible limit,
which however were performed in the context of the classical low Mach number regime. In particular, paper \cite{Li-Mu} considers the same penalisation for both pressure
gradients arising in the momentum equation; this authomatically implies that the two densities must be small perturbations of constant states and the singular
limit can be studied by resorting to somehow classical arguments.
Work \cite{F-K-N-Z}, instead, considers a sort of non-linear pressure law $P(\rho,\Theta) = \left(\rho\,\Theta\right)^\g$,
still penalised by a pre-factor $\veps^{-2}$ like in the classical low Mach number regime. Here, $\Theta$
plays the role of the temperature, which however is simply advected by the velocity fields (more precisely,
the quantity $\rho\,\Theta$ satisfies a continuity equation) and the thermal diffusion is neglected.
Importantly, and similarly to our strategy here, the authors in \cite{F-K-N-Z} need to resort to the transport equation for the temperature $\Theta$
(which has the same role of the quantity $b_\veps/\rho_\veps$ in this paper), but they need to impose
that its variations with respect to a given constant reference state are $O(\veps^2)$.
Compared to their assumption, here we require only $O(\veps)$ perturbations instead; actually, this is not really an assumption, but simply
follows from the condition imposed on the initial densities \eqref{ass:rho} and the analogous one (forced by the penalisation of the magnetic pressure)
for the initial magnetic functions $b_\veps$.

{\small
\section*{Acknowledgements}
F.F. has been partially supported by the project CRISIS (ANR-20-CE40-0020-01), operated by the French National Research Agency (ANR),
by the Basque Government through the BERC 2022-2025 program and by the Spanish State Research Agency through the BCAM Severo Ochoa excellence accreditation
CEX2021-001142. The author also aknowledges the support of the European Union through the COFUND program [HORIZON-MSCA-2022-COFUND-101126600-SmartBRAIN3].

}

\section{Problem formulation and main results} \label{s:result}

In this section, we specify our working setting and state our main result. We start by collecting, in Subsection \ref{ss:data}, the assumptions
on the initial data. In Subsection \ref{ss:weak-sol} we spend a few words on weak solutions to problem \eqref{eq:reduced-MHD}-\eqref{hyp:pressure}-\eqref{hyp:S}.
Finally, in Subsection \ref{ss:result} we present the statement of the main results of the paper, namely Theorems \ref{th:limit} and \ref{th:lim_Coriols}.

From now on, we adopt the following notation: given a normed space $X$ and a sequence $\big(f_\veps\big)_\veps\subset X$, we write
$\big(f_\veps\big)_\veps\Subset X$ if the sequence is bounded in $X$, meaning that there exists a constant $C>0$ such that
$\sup_{\veps}\left\|f_\veps\right\|_{X}\leq C$.

\subsection{Initial data} \label{ss:data}
In this work, we are interested in performing the limit $\veps\ra0^+$ in system \eqref{eq:reduced-MHD} for general \emph{ill-prepared} initial data.
More precisely, we are going assume that the initial data are small perturbations of a given static state, but we do not impose any special structure on those
perturbations. For simplicity, we are going to consider \emph{constant} static states; however, this is not the only possible choice.

Given any $\veps\in\,]0,1]$ fixed, we supplement system \eqref{eq:reduced-MHD} with the initial datum
\begin{equation} \label{eq:in-data}
 \big(\rho,u,b\big)_{|t=0}\,=\,\big(\rho_{0,\veps}, u_{0,\veps}, b_{0,\veps}\big)\,.
\end{equation}
We assume the following hypotheses on the family $\big(\rho_{0,\veps}, u_{0,\veps}, b_{0,\veps}\big)_{\veps\in\,]0,1]}$.

\paragraph*{Initial densities.} The initial densities $\rho_{0,\veps}$ are assumed to be positive measurable functions, \tsl{i.e.} $\rho_{0,\veps}\geq0$ a.e.
on $\Omega$, such that they are small perturbations of the static state $\oline{\rho}\in\R_+\setminus\{0\}$:
\[
 \rho_{0,\veps}\,=\,\oline\rho\,+\,\veps\,r_{0,\veps}\,,\qquad\qquad \mbox{ with }\qquad \big(r_{0,\veps}\big)_{\veps}\,\Subset\,L^2(\Omega)\cap L^\infty(\Omega)\,.
\]

Observe that, $\Omega$ being a bounded domain here, the $L^2$ condition is redundant. However, we impose it to stress that our study can be performed also
in more general domains with no essential modifications, \tsl{e.g.} in the case $\Omega=\R^3$ when $\k=0$ and $\Omega=\R^2\times\,]0,1[\,$ for $\k=1$. We refrain from
doing so here, only because we lack of a solid existence theory of finite energy weak solutions, however we do want to keep a general
point of view.

\paragraph*{Initial magnetic fields.} As for the initial magnetic fields $b_{0,\veps}$, they are assumed to be positive measurable functions as the density
functions, namely $b_{0,\veps}\geq0$ a.e. on $\Omega$; in addition, forced by the scaling in equations \eqref{eq:reduced-MHD},
we assume that they are small perturbations of some constant reference state
$\oline{b}>0$:
\[
b_{0,\veps}\,=\,\oline b\,+\,\veps\,\beta_{0,\veps}\,,\qquad\qquad \mbox{ with }\qquad \big(\beta_{0,\veps}\big)_{\veps}\,\Subset\,L^2(\Omega)\cap L^\infty(\Omega)\,.
\]
In fact, the sign of the constant $\oline b$ is not really important in our analysis. We take it positive just to be consistent with the interpretation
of system \eqref{eq:reduced-MHD} as a  of two-fluid model, in which case the unknown $b$ plays the role of a second density.

In the proof of our results, we will need to identify target profiles for the density
perturbations $\big(r_{0,\veps}\big)_\veps$ and magnetic field perturbations $\big(\beta_{0,\veps}\big)_\veps$.
Without loss of generality, we can assume that there exist $r_0$ and $\beta_0$, both belonging to $L^2(\Omega)\cap L^\infty(\Omega)$, such that
the whole sequences converge to those profiles:
\begin{equation}  \label{conv:r-beta_init}
 r_{0,\veps}\,\stackrel{*}{\rightharpoonup}\,r_0\qquad \mbox{ and }\qquad \beta_{0,\veps}\,\stackrel{*}{\rightharpoonup}\,\beta_0
 \qquad \qquad \mbox{ in } \qquad L^2(\Omega)\cap L^\infty(\Omega)\,.
\end{equation}

\paragraph*{Initial velocity fields.} 
The initial velocity fields $u_{0,\veps}$ are assumed to simply satisfy the uniform boundedness property
\[
\big(u_{0,\veps}\big)_\veps\,\Subset\,L^2(\Omega)\,.
\]
Without loss of generality, we can assume that there exists a velocity field $u_0\in L^2(\Omega)$ such that
\begin{equation} \label{conv:u_0}
 u_{0,\veps}\,\rightharpoonup\,u_0\qquad\qquad \mbox{ in }\qquad L^2(\Omega)\,.
\end{equation}
Again, we notice that it is not restrictive, at this stage, to assume that the whole sequence of $u_{0,\veps}$'s is converging to this profile.

\paragraph*{An additional key property.}

Observe that, as a consequence of the form of $\rho_{0,\veps}$ and $b_{0,\veps}$, we can compute
\[
\frac{b_{0,\veps}}{\rho_{0,\veps}}\,-\,\frac{\oline b}{\oline\rho}\,=\,\frac{\veps}{\oline\rho\,\rho_{0,\veps}}\,\left(\oline\rho\,\beta_{0,\veps}\,-\,
\oline b\,r_{0,\veps}\right)\,.
\]
Owing to the previous assumptions on the families $\big(r_{0,\veps}\big)_\veps$ and $\big(\beta_{0,\veps}\big)_\veps$, we can assume withtout loss of generality
that $\rho_{0,\veps}$ remains strictly positive for any $\veps\in\,]0,1]$: more precisely, there exists a constant $\rho_*>0$, independent of $\veps\in\,]0,1]$,
such that one has $\rho_{0,\veps}\geq\rho_*>0$ a.e. in $\Omega$.

Therefore, it follows that there exists a constant $C_0>0$, only depending on this constant $\rho_*$ and on $\oline\rho$ and $\oline b$, such that
\begin{equation} \label{est:b-r_in}
\forall\,\veps\in\,]0,1]\,,\qquad\qquad \frac{1}{\veps}\left|\frac{b_{0,\veps}}{\rho_{0,\veps}}\,-\,\frac{\oline b}{\oline \rho}\right|\,\leq\,C_0\qquad \mbox{ a.e. in }\ \Omega\,.
\end{equation}

\begin{rem} \label{r:key-prop}
 In the theory of two-fluid flows (see \tsl{e.g.} \cite{V-W-Y, N-P, Li-Sun, Li-Sun_JDE}), it is customary to assume the weaker condition
\[
\forall\,\veps\in\,]0,1]\,,\qquad\qquad 0\,<\,C_1\,\leq\,\frac{b_{0,\veps}}{\rho_{0,\veps}}\,\leq\,C_2\qquad \mbox{ a.e. in }\ \Omega\,.
\]
This assumption was also formulated in \cite{F-K-N-Z} in a study of a singular limit problem related to ours (see the explanations in Subsection \ref{ss:over}).
However, this assumption is not enough for our scopes; instead, we need to use the stronger condition \eqref{est:b-r_in}.

We stress the fact that 
\eqref{est:b-r_in} is a consequence of our assumptions, but it uses the form $\rho_{0,\veps} = \oline\rho + O(\veps)$ in a crucial way. This is one of the main
reasons evocated in the Introduction for imposing such structure on the initial density profiles.
\end{rem}

\subsection{Finite energy weak solutions} \label{ss:weak-sol}

The assumptions on the initial data $\big(\rho_{0,\veps}, u_{0,\veps}, b_{0,\veps}\big)_\veps$ having been clarified, let us give the definition of weak solutions
relevant for our study. As this definition is strongly based on the notion of \emph{finite energy}, let us begin by spending a few words on presenting the energy
functional associated to a solution of system \eqref{eq:reduced-MHD}.

First of all, we define the internal energy function $H$ (often dubbed \emph{pressure potential}) as the solution of the ODE
\[
\rho\,H'(\rho)\,-\,H(\rho)\,=\,P(\rho)\,.
\]
Notice that $H''(\rho)\,=\,P'(\rho)/\rho$, so, owing to assumption \eqref{hyp:pressure} on $P$, we deduce that $H$ is a convex function on $[0,+\infty[\,$.
From now on, we fix the choice
\begin{equation} \label{def:H}
 H(\rho)\,:=\,\rho\,\int^\rho_{\oline\rho}\frac{P(z)}{z^2}\,\dd z\,.
\end{equation}
We also immediately introduce the function
\begin{equation} \label{def:H_rel}
\mc H\left(\rho\,\big|\,\oline\rho\right)\,:=\,H(\rho)\,-\,H(\oline\rho)\,-\,H'(\oline\rho)\,\big(\rho-\oline\rho\big)
\end{equation}
to denote the \emph{Bregman divergence} associated to the convex function $H$. The function $\mc H\big(\rho\,|\,\oline\rho\big)$
is a measure of the density departures from the static value $\oline\rho$.

Observe that, starting from the relation $H''(\rho)\,=\,P'(\rho)/\rho$ and imposing the additional conditions $H(\oline\rho)=H'(\oline\rho)=0$,
one would find a modified function $\wtilde H$ with respect to $H$, which already measures the density variations from the value $\oline\rho$ and
is in fact equivalent to the Bregman divergence $\mc H\big(\rho\,\big|\,\oline\rho\big)$.
For instance, in the classical case of Boyle's pressure law $P(\rho)\,=\,\rho^\g/\g$, with $\g>1$, we would find the explicit expression
\[
\wtilde H(\rho)\,=\,\wtilde H_\g(\rho)\,=\,\frac{1}{\g(\g-1)}\,\Big(\rho^\g\,-\,\oline\rho^\g\,-\,\g\,\oline\rho^{\g-1}\,\big(\rho-\oline\rho\big)\Big)\,=\,
\frac{1}{(\g-1)}\,\mc H\big(\rho\,\big|\,\oline\rho\big)\,.
\]
Applying the same argument (with $\g=2$) to the magnetic pressure function $P_{\rm magn}(b)\,=\,b^2/2$, we find the simple formula
\begin{equation} \label{def:B}
\mc B\big(b\,\big|\,\oline b\big)\,:=\,\wtilde H_2(b)\,=\,\frac{1}{2}\,\left(b-\oline b\right)^2\,.
\end{equation}

Then, we can define the energy of a triplet $\big(\rho,u,b\big)$ as
\begin{equation} \label{def:E}
\mc E\big(\rho,u,b\big)\,:=\,\frac{1}{2}\int_{\Omega}\rho\,|u|^2\,\dx\,+\,\int_{\Omega}\mc H\big(\rho\,\big|\,\oline\rho\big)\,\dx\,+\,
\frac{1}{\veps^2}\int_\Omega \mc B\big(b\,\big|\,\oline b\big)\,\dx\,.
\end{equation}
Whenever $\big(\rho,u,b\big)$ is a solution to \eqref{eq:reduced-MHD} related to some initial datum $\big(\rho_0,u_0,b_0\big)$ and defined on $\R_+\times\Omega$,
for $t>0$ we set
\[
\mc E\big(\rho,u,b\big)(t)\,:=\,\mc E\big(\rho(t),u(t),b(t)\big)\,.
\]
For simplicity, we will often use the notation $\mc E\big(\rho,u,b\big)(0)\,:=\,\mc E\big(\rho_0,u_0,b_0\big)$.

After these preliminaries, we can give the definition of \emph{finite energy weak solution} to system \eqref{eq:reduced-MHD}.
We adapt the definition from to \cite{Li-Sun_JDE} to our framework.

\begin{defin} \label{def:weak}
Let $T>0$ be given and $\Omega=\T^3$.
Let $\big(\rho_0,u_0,b_0\big)$ be a triplet of initial data satisfying the assumptions formulated in Subsection \ref{ss:data}
(for a fixed $\veps>0$), with in addition $\rho_0\geq0$ and $b_0\geq0$.

A triplet $\big(\rho,u,b\big)$ is said to be a \emph{finite energy weak solution} to system \eqref{eq:reduced-MHD}
on $[0,T[\,\times\Omega$, related to the initial datum $\big(\rho_0,u_0,b_0\big)$, if the following conditions are satisfied:
\begin{itemize}
 \item $\rho\geq0$ and $\rho-\oline\rho\in L^\infty\big([0,T[\,;L^{\min\{2,\g\}}(\Omega)\big)$;
 \item $b\geq0$ and $b-\oline b\in L^\infty\big([0,T[\,;L^{2}(\Omega)\big)$;
 \item $\sqrt{\rho}\,u\in L^\infty\big([0,T[\,;L^{2}(\Omega)\big)$ and $\nabla u\in L^2\big([0,T[\,;L^2(\Omega)\big)$;
 \item the equations in \eqref{eq:reduced-MHD} are satisfied in the sense of $\mc D'\big([0,T[\,\times\Omega\big)$ and, in addition, the equations for $\rho$ and $b$
 hold true in a renormalised sense: for any $h\in C^1(\R)$ such that $h'(z)=0$ for sufficiently large values of $z$, 
one has, for $\mf z\in\{\rho,b\}$, the relation
\[
 \d_th(\mf z)\,+\,\div\big(h(\mf z)\,u\big)\,+\,\left(h'(\mf z)\,\mf z\,-\,h(\mf z)\right)\,\div u
 \,=\,0\qquad\quad \mbox{ in }\qquad \mc D'\big([0,T[\,\times\Omega\big)\,;
\]
 \item the energy inequality
\[
\mc E\big(\rho,u,b\big)(t)\,+\,\int^t_0\int_\Omega\mbb S(\D u):\D u\,\dx\,\dd\tau\,\leq\,\mc E\big(\rho_0,u_0,b_0\big)
\]
holds true for a.e. $t\in[0,T[\,$.
\end{itemize}

If the above conditions hold with $T=+\infty$, then the solution is said to be \emph{global}.
\end{defin}

Strictly speaking, the requirement $b\geq 0$ is not present in \cite{Li-Sun_JDE}, as in that work the variable $b$ represented the magnetic field
and no sign condition was formulated on the initial datum.
Owing to our assumption on $b_{0}\geq0$, however, we can easily guarantee that $b\geq0$, precisely as it is done for the density $\rho$.

The existence of finite energy weak solutions to system \eqref{eq:reduced-MHD} was proved in the same work \cite{Li-Sun_JDE} in a two-dimensional setting,
and for the case of Boyle's pressure laws $P(\rho)\,=\,\rho^\g$, with $\g>1$.
Here we report the adaptation of that result to a $3$-D domain. Notice that, in this case, no further restriction on the adiabatic exponent $\g>1$
appears, thanks to the special structure of the two-fluid type system \eqref{eq:reduced-MHD} and the interplay between $b$ and $\rho$.

In \cite{Li-Sun_JDE}, the authors proved the existence of global in time finite energy weak solutions to the reduced MHD model in a smooth bounded domain
of $\R^2$; in \cite{Li-Sun}, the same authors extended their result to heat-conducting fluids. Strictly speaking, the results were obtained for classical
barotropic pressure laws $P(\rho) = A\,\rho^\g$, however the analysis can be accomodated (by performing fairly standard modifications) to deal with more general
pressure lwas of the form of \eqref{hyp:pressure}. Analogously, the two-dimensional assumption was only needed to get the special form of the sytem, while a similar
study could be arried out in any bounded domain of $\R^3$.

\begin{thm} \label{th:existence}
Let $\Omega=\T^3$ and take $\g>1$ in \eqref{hyp:pressure}. Assume that the initial datum $\big(\rho_0,u_0,b_0\big)$ satisfy the conditions
\[
0\,<\,K_0\,\leq\,\rho_0\,,\,b_0\,\leq\,K_1\qquad\qquad \mbox{ and }\qquad\qquad u_0\,\in\,L^2(\Omega)\,,
\]
for two suitable positive constants $K_0\leq K_1$.

Then, there exists one global finite energy weak solution $\big(\rho,u,b\big)$ to system \eqref{eq:reduced-MHD}, in the sense of Definition \ref{def:weak}.
In addition, the quantity $Z$, defined as
\[
 Z\,=\,\left\{\begin{array}{ll}
               \dfrac{b}{\rho} &\qquad \mbox{ if }\quad \rho>0\,, \\[2ex]
               1 &\qquad \mbox{ if }\quad \rho=0\,,
              \end{array}
              \right.
\]
belongs to $L^\infty\big(\R_+\times\Omega\big)$ and satisfies (in the weak sense) the transport equation
\[
 \d_tZ\,+\,u\cdot\nabla Z\,=\,0\,,\qquad\qquad Z_{|t=0}\,=\,\frac{b_0}{\rho_0}\,.
\]
\end{thm}

To conclude, we point out that the last sentence concerning the function $Z$ is not contained in the statement of \cite{Li-Sun_JDE} (see Theorem 1.1 and Remark 1.1
therein). However, it can be easily deduced by following the steps of the proof and using the strong convergence properties of $\rho$ and $b$
at the level of the smooth approximate solutions (alternatively, one can use Lemma 6.1 of \cite{F-K-N-Z}).

\subsection{Main results} \label{ss:result}

After the previous preparatory work, we can now state the main results of this paper.

Let us start by considering the case $\k=0$ (no Coriolis force). Because of technical reasons arising in Subsection \ref{ss:bounds_energy},
we need to assume the adiabatic exponent $\g$ to satisfy $\g>3/2$ here (as our domain is three-dimensional). The assumption $\g>1$ would be enough
in the case we were working on the two-dimensional torus, or in case we had imposed some smallness condition on the constant $C_0$
appearing in condition \eqref{est:b-r_in} (see Remark \ref{r:two-fluid} for more comments about this).

\begin{thm} \label{th:limit}
Consider system \eqref{eq:reduced-MHD} set on the domain $\Omega=\T^3$.
Assume that $\g>3/2$ in \eqref{hyp:pressure}.
Consider a sequence of initial data $\big(\rho_{0,\veps}, u_{0,\veps}, b_{0,\veps}\big)_\veps$
verifying the assumptions listed in Subsection \ref{ss:data}. Let $\big(\rho_\veps,u_\veps,b_\veps\big)_\veps$ be a sequence of global in time finite energy weak solutions
related to those data, according to Theorem \ref{th:existence}.
Let $u_0\in L^2(\Omega)$ be identified as in \eqref{conv:u_0}.

Then, in the limit $\veps\to0$, one has the strong convergence properties
\[
 \rho_\veps\,\longrightarrow\,\oline\rho\qquad \mbox{ and }\qquad b_\veps\,\longrightarrow\,\oline b\qquad\qquad \mbox{ in }\qquad 
L^\infty\big(\R_+;L^{\min\{2,\g\}}(\Omega)\big)\,.
\]
In addition, there exists a vector field $u\in L^\infty\big(\R_+;L^2(\Omega)\big)\cap L^2_\loc\big(\R_+;H^1(\Omega)\big)$,
with $\div u =0$, such that, up to the extraction of a suitable
subequence, in the limit $\veps\to0$ one has $u_\veps\,\stackrel{*}{\rightharpoonup}\,u$ in $L^2_\loc\big(\R_+;H^1(\Omega)\big)$
and $\sqrt{\rho_\veps}\,u_\veps\,\stackrel{*}{\rightharpoonup}\,\sqrt{\oline\rho}\,u$ in $L^\infty\big(\R_+;L^2(\Omega)\big)$.
Finally, the vector field $u$ solves, in the weak sense, the incompressible Navier-Stokes system
\[
 \left\{\begin{array}{l}
         \d_tu\,+\,\div\big(u\otimes u\big)\,+\,\nabla\Pi\,-\,\frac{\mu}{\oline\rho}\,\Delta u\,=\,0 \\[1ex]
         \div u\,=\,0
        \end{array}
\right.
\]
related to the initial datum $u_0$, for a suitable pressure gradient $\nabla\Pi$.
\end{thm}

Before moving on, let us observe the following important point. 

\begin{rem} \label{r:beyond}
Although this does not appear in the statement, for proving Theorem \ref{th:limit} it will be crucial
not only to use that $b_\veps = \oline b + O(\veps)$ (which follows from the scaling) and $\rho_\veps = \oline\rho + O(\veps)$ (which is
the induced low Mach number effect mentioned in the Introduction), but also to \emph{identify} the limit profiles of those $O(\veps)$ perturbations.

This is very different from the classical case of incompressible or anelastic limits (see \tsl{e.g.} \cite{L-Masm, F-N, F-K-N-Z, F-Z} and references therein),
in which in fact such a target profile not only is not needed, but in fact \emph{cannot} be identified. In our work,
this is made possible by the use of an additional transport structure underlying our system of equations (this is strictly related to the transport
equation for $Z$ appearing in Theorem \ref{th:existence}). We refer to the beginning of Subsection \ref{ss:hidden} for more explanations about this point.
\end{rem}

\medbreak
We then consider the case $\k=1$, where the incompressible limit combines with the fast rotation limit. 
In order to state our result, some preliminary notation is in order. For any vector $V = (V_1,V_2,V_3)\in\R^3$, we denote by $V_h$ the two-dimensional vector
composed of only the first two components of $V$, namely $V_h = (V_1,V_2)$. Similarly, we denote by $\nabla_h$ and $\divh$ the usual
differential operators acting only on the variables $x_h=(x_1,x_2)$: for instance,
$\divh V_h = \d_1V_1 + \d_2V_2$. Moreover, given a two-dimensional vector $v=(v_1,v_2)\in\R^2$, we denote by $v^\perp$ its rotation of angle $\pi/2$, that is
$v^\perp = (-v_2, v_1)$. Analogously, we define the rotated horizontal gradient $\nabla_h^\perp = (-\d_2,\d_1)$.
Finally, for any function $f = f(t,x)$, we denote $\lan f\ran = \lan f\ran(t,x_h)$ its vertical average, namely we define
$\lan f\ran(t,x_h)\,:=\,\int_0^1f(t,x_h,z)\,\dd z$.

In the case $\k=1$, we have the following convergence result. 

\begin{thm} \label{th:lim_Coriols}
Consider system \eqref{eq:reduced-MHD} set on the domain $\Omega=\T^2\times\,]0,1[\,$, supplemented with the complete-slip boundary conditions \eqref{bc:compl-slip}.
Assume that \eqref{hyp:pressure} holds true with $\g>3/2$.
Consider a sequence of initial data $\big(\rho_{0,\veps}, u_{0,\veps}, b_{0,\veps}\big)_\veps$
verifying the assumptions listed in Subsection \ref{ss:data}. Let $\big(\rho_\veps,u_\veps,b_\veps\big)_\veps$ be a sequence of global in time finite energy weak solutions
related to those data, according to Theorem \ref{th:existence}.
Let $u_0\in L^2(\Omega)$ and $r_0$, $\beta_0\,\in L^2(\Omega)\cap L^\infty(\Omega)$ be identified as, respectively, in \eqref{conv:u_0} and \eqref{conv:r-beta_init}.
For any $\veps>0$ fixed, define the quantities
\[
 r_\veps\,:=\,\frac{1}{\veps}\,\big(\rho_\veps-\oline\rho\big)\,,\qquad \beta_\veps\,:=\,\frac{1}{\veps}\,\big(b_\veps-\oline b\big)\,,
\qquad \alpha_\veps\,:=\,\oline\rho\,\beta_\veps\,-\,\oline b\,r_\veps\,.
\]
For convenience of notation, let us also set $c_0\,:=\,\frac{\oline b}{\oline\rho}$ and ${m}\,:=\,\min\{2,\g\}$.

Then, up to the extraction of a suitable subsequence and for suitable target profiles belonging to the respective spaces, one has the following convergence properties
in the limit $\veps\to0$:
\begin{itemize}
 \item $r_\veps\,\stackrel{*}{\rightharpoonup}\,r$ in the space $L^\infty\big(\R_+;L^{m}(\Omega)\big)$;
 \item $\beta_\veps\,\stackrel{*}{\rightharpoonup}\,\beta$ in the space $L^\infty\big(\R_+;L^{2}(\Omega)\big)$, where $\beta = \beta(t,x_h)$ only depends
 on the horizontal variables;
 \item $\alpha_\veps\,\stackrel{*}{\rightharpoonup}\,\alpha\,:=\,\oline\rho\,\beta\,-\,\oline b\,r$ in the space
$L^\infty\big(\R_+;L^{m}(\Omega)\big)$, and the convergence is strong in the space
$L^\infty\big(\R_+;W^{-1,m}(\Omega)\big)$;
\item $u_\veps\,\stackrel{*}{\rightharpoonup}\,c_0\,\big(\nabla^\perp_h\beta,0\big)$ in the space
$L^2_\loc\big(\R_+;H^1(\Omega)\big)$.
\end{itemize}
In addition, the couple $\big(\alpha,\beta\big)$ is a weak solution to the following quasi-geostrophic type system,
\[
\left\{ \begin{array}{l}
         \d_t\alpha\,+\,c_0\,\nabla^\perp_h\beta\cdot\nabla_h\alpha\,=\,0 \\[1ex]
         \d_t\left(\frac{1}{c_0}\,\beta\,-\,\oline\rho\,c_0\,\Delta_h\beta\,-\,\frac{1}{\oline\rho\,c_0}\,\lan\alpha\ran\right)\,-\,
         \oline\rho\,c_0^2\,\nabla^\perp_h\beta\cdot\nabla_h\Delta_h\beta\,+\,c_0\,\Delta_h^2\beta\,=\,0\,,
        \end{array}
\right.
\]
related to the initial data 
\begin{align*}
 \alpha_{|t=0}\,=\,\oline\rho\,\beta_0\,-\,\oline b\,r_0\qquad \mbox{ and }\qquad
 \left(\frac{1}{c_0}\,\beta\,-\,\oline\rho\,c_0\,\Delta_h\beta\,-\,\frac{1}{\oline\rho\,c_0}\,\lan\alpha\ran\right)_{|t=0}\,=\,
 \lan r_0\ran-\,\oline\rho\,\lan\omega_0\ran\,,
\end{align*}
where we have defined $\o_0\,:=\,\curlh u_{0,h}\,=\,\d_1u_{0,2}\,-\,\d_2u_{0,1}$.
\end{thm}

\begin{rem} \label{r:closed-syst}
The second equation appearing in the limit system can be also rewritten as
\[
\d_t\big(\lan r\ran\,-\,\oline\rho\,c_0\,\Delta_h\beta\big)\,-\,\oline\rho\,c_0^2\,\nabla^\perp_h\beta\cdot\nabla_h\Delta_h\beta\,+\,c_0\,\Delta_h^2\beta\,=\,0\,,
\]
which may be more familiar to the reader expert in singular limits for fast rotating fluids (see \tsl{e.g.} \cite{F-G-N, DS-F-S-WK, F_PhysD}).
Here, we have exploited the definition of $\alpha$ in order to write a closed system for the couple of unknowns
$\big(\alpha,\beta\big)$. Indeed the system is closed, as, by taking the vertical averages of the first equation and using that $\beta$ does not depend
on $x_3$, one gets an equation for $\lan\alpha\ran$:
\[
\d_t\lan\alpha\ran\,+\,c_0\,\nabla^\perp_h\beta\cdot\nabla_h\lan\alpha\ran\,=\,0\,,
\]
with initial datum $\lan\alpha\ran_{|t=0}\,=\,\oline\rho\,\lan\beta_0\ran\,-\,\oline b\,\lan r_0\ran$.
\end{rem}

Finally, we point out that a similar observation as Remark \ref{r:beyond} applies also to the fast rotation limit. Here, the needed information
on the limit profiles for the $O(\veps)$ perturbations of the static states $\oline b$ and $\oline\rho$ is in fact contained in the identification
of the quantity $\alpha_\veps$ and its limit $\alpha$.

\medbreak
The rest of this work is devoted to the proof of Theorems \ref{th:limit} and \ref{th:lim_Coriols}.

\section{Uniform bounds and weak convergence} \label{s:uniform}

We collect here the uniform bounds for the family of weak solutions $\big(\rho_\veps,u_\veps,b_\veps\big)_\veps$. In the first part of the section,
we focus on esimtates which come from the energy inequality.
In the second part of the section, we get important $L^\infty$ estimates from the transport equation satisfied by the quantity $Z_\veps\,:=\,b_\veps/\rho_\veps$.
By the use of those uniform bounds,  in Subsection \ref{ss:weak_conv} we deduce first (weak) convergence properties
for the sequence $\big(\rho_\veps,u_\veps,b_\veps\big)_\veps$. Finally, in Subsection \ref{ss:hidden} we prove a key compactness result on a special
quantity linking the $b_\veps$'s and the $\rho_\veps$'s.

\subsection{Estimates coming from the energy inequality} \label{ss:bounds_energy}

Given a family of global in time finite energy weak solutions  $\big(\rho_\veps,u_\veps,b_\veps\big)_{\veps}$, by Definition \ref{def:weak} we know that,
for any $\veps\in\,]0,1]$ fixed and for almost any $T>0$, one has
\begin{equation} \label{ub:energy}
\mc E\big(\rho_\veps,u_\veps,b_\veps\big)(T)\,+\,\int^T_0\int_\Omega\mbb S(\D u_\veps):\D u_\veps\,\dx\,\dd t\,\leq\,\mc E\big(\rho_{0,\veps},u_{0,\veps},b_{0,\veps}\big)\,.
\end{equation}
Owing to the assumptions formulated on the family of initial data
$\big(\rho_{0,\veps},u_{0,\veps},b_{0,\veps}\big)_{\veps}$  (see Subsection \ref{ss:data}), we deduce that there exists a constant $K>0$ such that
\[
 \sup_{\veps\in\,]0,1]}\mc E\big(\rho_{0,\veps},u_{0,\veps},b_{0,\veps}\big)\,\leq\,K\,.
\]
Therefore, from the energy inequality \eqref{ub:energy} we can extract suitable uniform bounds for the family of weak solutions.
Recall that we write $\big(f_\veps\big)_\veps\Subset X$ to denote that the sequence $\big(f_\veps\big)_\veps\subset X$ is \emph{bounded} in the Banach space $X$.

\medbreak
First of all, we consider the velocity fields. We have
\begin{equation} \label{est:u_en}
\left(\sqrt{\rho_\veps}\,u_\veps\right)_{\veps}\,\Subset\,L^\infty\big(\R_+;L^2(\Omega)\big)\qquad\qquad \mbox{ and }\qquad\qquad
\left(\nabla u\right)_\veps\,\Subset\,L^2\big(\R_+;L^2(\Omega)\big)\,.
\end{equation}

Next, we consider the magnetic functions $b_\veps$. From the last term in the definition of the energy functional $\mc E\big(\rho_\veps,u_\veps,b_\veps\big)$, keep in mind
relations \eqref{def:E} and \eqref{def:B}, it is easy to see that
\begin{equation} \label{eq:b_decomp}
b_\veps\,=\,\oline b\,+\,\veps\,\beta_\veps\,,\qquad\qquad \mbox{ with }\qquad \left(\beta_\veps\right)_\veps\,\Subset\,L^\infty\big(\R_+;L^2(\Omega)\big)\,.
\end{equation}

Now we consider the density functions. Following a nowadays classical approach (see \tsl{e.g.} the book \cite{F-N}),
we introduce the notion of essential and residual sets, $\Omega^\veps_\ess$ and $\Omega^\veps_\res$ respectively:
\[
 \Omega^\veps_\ess(t)\,:=\,\left\{x\in\Omega\;\Big|\quad \frac{\oline\rho}{2}\,\leq\,\rho(t,x)\,\leq\,2\,\oline\rho\,\right\}\,,\qquad\qquad
 \Omega^\veps_\res(t)\,:=\,\Omega\setminus\Omega^\veps_\ess(t)\,.
\]
Let us denote by $\mds1_A$ the characteristic function of a set $A\subset\Omega$, namely $\mds1_A(x)=1$ if $x\in A$ and $\mds1_A(x)=0$ if $x\in\Omega\setminus A$.
Then, for any function $f$ defined on $\R_+\times\Omega$, we set
\[
[f]_\ess\,:=\,f\,\mds1_{\Omega^\veps_\ess}\qquad\qquad \mbox{ and }\qquad\qquad
[f]_\res\,:=\,f\,\mds1_{\Omega^\veps_\res}\,=\,f\,\left(1-\mds1_{\Omega^\veps_\ess}\right)\,.
\]
By using classical computations (see \tsl{e.g.} \cite{F_PhysD}), from the density term in the energy functional and the uniform bound in \eqref{ub:energy}
one can deduce that
\begin{equation} \label{eq:dens_decomp}
\rho_\veps\,=\,\oline\rho\,+\,\s_\veps\,, 
\end{equation}
where one has
\begin{equation} \label{ub:s_eps}
\left(\big[\s_\veps\big]_\ess\right)_\veps\,\Subset\,L^\infty\big(\R_+;L^2(\Omega)\big)\qquad\qquad \mbox{ and }\qquad\qquad
\left(\big[\s_\veps\big]_\res\right)_\veps\,\Subset\,L^\infty\big(\R_+;L^\g(\Omega)\big)\,.
\end{equation}

Next, in a standard way, see \tsl{e.g.} \cite{F-G-N} and \cite{F_PhysD}, combining \eqref{est:u_en} with \eqref{eq:dens_decomp} and \eqref{ub:s_eps} and using (in the
three-dimensional setting) that $\g>3/2$, one can get
\begin{equation} \label{ub:u_H^1}
\left(u_\veps\right)_\veps\,\Subset\,L^2_{\rm loc}\big(\R_+;H^1(\Omega)\big)\,.
\end{equation}

Observe that, in our context, we cannot use the classical argument in two-fluid systems to improve the integrability of $\rho_\veps$ thanks to $b_\veps$,
thus getting rid of the assumption $\g>3/2$. See Remark \ref{r:two-fluid} below for more comments about this point.

\subsection{Estimates coming from the transport equation} \label{ss:bounds_transport}

Decomposition \eqref{eq:dens_decomp} is a main obstacle when computing the asymptotics of system \eqref{eq:reduced-MHD} for $\veps\ra0^+$,
as it does not entail any smallness for the density perturbation functions $\s_\veps$. On the contrary, as $b_\veps-\oline b=O(\veps)$, the
third equation of \eqref{eq:reduced-MHD} imposes that (roughly speaking) $\div u_\veps$ has to be small; combining this with the mass conservation equation
in \eqref{eq:reduced-MHD}, we thus expect to gain some smallness also for the functions $\s_\veps$.

Justifying those heuristics and finding quantitative smallness bounds for the family $\big(\s_\veps\big)_\veps$ is the main goal of this subsection.
For obtaining these properties, the crucial point is to exploit the transport equation satisfied (according to the last statement in Theorem \ref{th:existence}),
by the quotient $b_\veps/\rho_\veps$.
More precisely, let us introduce the functions
\[
\mc Z_{0,\veps}\,:=\,\frac{1}{\veps}\left(\frac{b_{0,\veps}}{\rho_{0,\veps}}\,-\,\frac{\oline b}{\oline\rho}\right)\qquad\qquad \mbox{ and }\qquad\qquad
\mc Z_{\veps}\,:=\,
\left\{\begin{array}{ll}
        \dfrac{1}{\veps}\left(\dfrac{b_{\veps}}{\rho_{\veps}}\,-\,\dfrac{\oline b}{\oline\rho}\right) & \quad \mbox{ if }\; \rho_\veps>0\,, \\[2ex]
        1 & \quad \mbox{ if }\; \rho_\veps=0\,.
       \end{array}
\right.
\]
Then, for any $\veps>0$ fixed, $\mc Z_\veps$ satisfies the transport problem
\[
\d_t\mc Z_\veps\,+\,u_\veps\cdot\nabla\mc Z_\veps\,=\,0\,,\qquad\qquad \big(\mc Z_\veps\big)_{|t=0}\,=\,\mc Z_{0,\veps}\,.
\]

As, by assumption \eqref{est:b-r_in}, we have that $\big(\mc Z_{0,\veps}\big)_{\veps}\Subset L^\infty(\Omega)$,
by transport theory we deduce that such $L^\infty$ bound is preserved, uniformly with respect to $\veps\in\,]0,1]$:
\begin{equation} \label{est:b-r}
\forall\,\veps\in\,]0,1]\,,\qquad\qquad \frac{1}{\veps}\left|\frac{b_{\veps}}{\rho_{\veps}}\,-\,\frac{\oline b}{\oline\rho}\right|\,\leq\,C_0\qquad \mbox{ a.e. in }\ \R_+\times\Omega\,.
\end{equation}
In other words, one has
\begin{equation*} 
 \big(\mc Z_\veps\big)_{\veps}\,\Subset\,L^\infty\big(\R_+\times\Omega\big)\,.
\end{equation*}

Using \eqref{eq:b_decomp} and \eqref{eq:dens_decomp}, from the previous relation simple computations yield
\[
\left|\oline\rho\,\beta_\veps\,-\,\oline b\,\frac{\s_\veps}{\veps}\right|\,\leq\,C_0\,\oline\rho\,\rho_\veps\,,
\]
which implies, together with \eqref{ub:s_eps} and the fact that $\Omega=\T^3$ is bounded, that
\[
\left(\left[\oline\rho\,\beta_\veps\,-\,\oline b\,\frac{\s_\veps}{\veps}\right]_\ess\right)_\veps\,\Subset\,L^\infty\big(\R_+;L^{2}(\Omega)\big)\qquad \mbox{ and }\qquad 
\left(\left[\oline\rho\,\beta_\veps\,-\,\oline b\,\frac{\s_\veps}{\veps}\right]_\res\right)_\veps\,\Subset\,L^\infty\big(\R_+;L^{\g}(\Omega)\big)\,.
\]
In turn, after setting
\[
 r_\veps\,:=\,\frac{\s_\veps}{\veps}\,,
\]
from the previous uniform bounds and \eqref{eq:b_decomp}, we discover that we can indeed reinforce \eqref{eq:dens_decomp} to the decomposition
\begin{equation} \label{eq:dens_dec-2}
 \rho_\veps\,=\,\oline\rho\,+\,\veps\,r_\veps\,,
\end{equation}
where $\big(r_\veps\big)_\veps$ satifies the properties
\begin{equation} \label{ub:r_eps}
 \left(\left[r_\veps\right]_\ess\right)_\veps\,\Subset\,L^\infty\big(\R_+;L^2(\Omega)\big)\qquad \mbox{ and }\qquad
 \left(\left[r_\veps\right]_\res\right)_\veps\,\Subset\,L^\infty\big(\R_+;L^{\min\{2,\g\}}(\Omega)\big)\,.
\end{equation}

\begin{rem} \label{r:two-fluid}
Differently from what is usually assumed in mathematical studies on two-fluid flows (keep in mind Remark \ref{r:key-prop} above),
here we do not know that $\frac{\oline b}{\oline\rho}-C_0>0$, so we cannot upgrade the integrability properties of $\beta_\veps$ and $r_\veps$ to
the value $\max\{2,\g\}$, and we must stick to the above values of the integrability indices.
\end{rem}

\subsection{First convergence properties} \label{ss:weak_conv}

We can use the uniform boundedness properties established in Subsections \ref{ss:bounds_energy} and \ref{ss:bounds_transport} to deduce important
weak convergence (up to a suitable extraction, which will be omitted for simplicity of presentation) of the family of solutions
$\big(\rho_\veps,u_\veps,b_\veps\big)$. In this subsection, it is convenient to resort to the notation, introduced
in Theorem \ref{th:lim_Coriols},
\[
 m\,:=\,\min\{2,\g\}\,.
\]

First of all, owing to the decompositions \eqref{eq:b_decomp} and \eqref{eq:dens_dec-2}, we get, in the limit $\veps\to0$, the strong convergence properties
\begin{align}
\label{conv:dens-b}
&b_\veps\,\longrightarrow\,\oline b\qquad\qquad \mbox{ in }\qquad L^\infty\big(\R_+;L^2(\Omega)\big) \\
\nonumber
\mbox{ and }\quad &\rho_\veps\,\longrightarrow\,\oline\rho\qquad\qquad \mbox{ in }\qquad L^\infty\big(\R_+;L^{m}(\Omega)\big)\,,
\end{align}
with actually an explicit rate $O(\veps)$.

At the same time, using \eqref{eq:b_decomp} again together with \eqref{ub:r_eps}, we deduce the existence of functions
$\beta\in L^\infty\big(\R_+;L^2(\Omega)\big)$ and $r\in L^\infty\big(\R_+;L^{m}(\Omega)\big)$ such that
\begin{equation} \label{conv:b-r_weak}
  \beta_\veps\,\stackrel{*}{\rightharpoonup}\,\beta\quad \mbox{ in }\; L^\infty\big(\R_+;L^2(\Omega)\big)\qquad \mbox{ and }\qquad
 r_\veps\,\stackrel{*}{\rightharpoonup}\,r\quad \mbox{ in }\; L^\infty\big(\R_+;L^{m}(\Omega)\big)\,.
\end{equation}
By the same token, using \eqref{ub:u_H^1}, we find the existence of a profile $u\in L^2_\loc\big(\R_+;H^1(\Omega)\big)$ such that
\begin{equation} \label{conv:u_weak-H}
\forall\,T>0\,,\qquad\qquad u_\veps\,\rightharpoonup\,u\qquad \mbox{ in }\quad L^2\big([0,T];H^1(\Omega)\big)\,.
\end{equation}

\begin{rem} \label{r:conv_u}
Owing to the uniform bounds \eqref{est:u_en} and the decomposition \eqref{eq:dens_dec-2} for the density functions, by uniqueness of the weak limit
it is easy to see that, in fact, the target velocity field $u$ actually satisfies
$u\in L^\infty\big(\R_+;L^2(\Omega)\big)$, with $\nabla u\in L^2\big(\R_+;L^2(\Omega)\big)$ and that one has the weak convergences
$\sqrt{\rho_\veps}\,u_\veps\,\stackrel{*}{\rightharpoonup}\,\sqrt{\oline\rho}\,u$ in $L^\infty\big(\R_+;L^2(\Omega)\big)$ and
$\nabla u_\veps\,\rightharpoonup\,\nabla u$ in $L^2\big(\R_+;L^2(\Omega)\big)$.
\end{rem}

Thanks to the previous convergence results, we can already derive some static properties for the target profiles.
First of all, let us consider the weak formulation of the equation for $b_\veps$:
for any $\vphi\in C^\infty_0\big([0,T[\,\times\Omega\big)$, for some fixed time $T>0$, one has
\begin{align*}
 -\int_0^T\int_{\Omega}b_\veps\,\d_t\vphi\,\dx\,\dd t\,-\,\int^T_0\int_{\Omega}b_\veps\,u_\veps\cdot\nabla\vphi\,\dx\,\dd t\,=\,\int_\Omega b_{0,\veps}\,\vphi(0,\cdot)\,\dx\,.
\end{align*}
Using \eqref{conv:dens-b} and \eqref{conv:u_weak-H}, it is easy to pass to the limit in the above relation and get that
\[
 \oline b\,\int^T_0\int_{\Omega}\,u\cdot\nabla\vphi\,\dx\,\dd t\,=\,0
\]
for any test-function $\vphi$ as above. Therefore, we deduce that
\begin{equation} \label{constr:div}
 \div u\,=\,0\qquad\qquad \mbox{ a.e. in }\ \R_+\times\Omega\,.
\end{equation}

To proceed further, we must analyse the weak formulation of the momentum equation: given a test-function $\psi\in C^\infty_0\big([0,T[\,\times\Omega;\R^3\big)$,
for some time $T>0$, 
we have the equation
\begin{align}
\label{eq:u_weak}
&-\int^T_0\int_{\Omega}\rho_\veps\,u_\veps\cdot\d_t\psi\,\dx\,\dd t\,-\,\int^T_0\int_{\Omega}\rho_\veps\,u_\veps\otimes u_\veps:\nabla\psi\,\dx\,\dd t \\
\nonumber
&\qquad -\int^T_0\int_{\Omega}\left(P(\rho_\veps)\,+\,\frac{1}{2\,\veps^2}\,b_\veps^2\right)\,\div\psi\,\dx\,\dd t\,+\,
\frac{\k}{\veps}\int^T_0\int_\Omega e_3\times\rho_\veps\,u_\veps\cdot\psi\,\dx\,\dd t \\
\nonumber
&\qquad\qquad +\int^T_0\int_{\Omega}\mbb S(\D u_\veps):\D\psi\,\dx\,\dd t\,=\,\int_\Omega\rho_{0,\veps}\,u_{0,\veps}\cdot\psi(0,\cdot)\,\dx\,.
\end{align}

Consider now equation \eqref{eq:u_weak} where we use $\veps\psi$ instead of $\psi$ as a test-function.
Observe that $P(\rho_\veps)$ belongs $L^\infty\big([0,T];L^1(\Omega)\big)$ thanks to the bounds established in Subsection \ref{ss:bounds_energy}.
Using similarly the other uniform bounds, we can easily check that all the terms converge to $0$, except for the terms
depending on $b_\veps^2$ and on the Coriolis term $e_3\times\rho_\veps\,u_\veps$. Therefore, we obtain that
\begin{equation} \label{eq:constr-limit}
\lim_{\veps\to0}\left( -\int^T_0\int_{\Omega}\frac{1}{2\,\veps}\,b_\veps^2\,\div\psi\,\dx\,\dd t\,+\,\k\int^T_0\int_\Omega e_3\times\rho_\veps\,u_\veps\cdot\psi\,\dx\,\dd t
\right)\,=\,0\,.
\end{equation}

Let us carefully study the two terms appearing in \eqref{eq:constr-limit}. 
Firstly, we focus on the first integral. Observe that, resorting to the definition \eqref{def:B} and the decomposition \eqref{eq:b_decomp}, we can write
\begin{align} \label{eq:constr-equal}
\frac{1}{2\,\veps}\,b_\veps^2\,=\,\frac{\veps}{\veps^2}\,\mc B\big(b_\veps\,\big|\,\oline b\big)\,+\,
\frac{1}{2\,\veps}\,\oline b^2\,+\,\oline b\,\beta_\veps\,,
\end{align}
where, thanks to the energy balance \eqref{ub:energy}, one has
\begin{equation} \label{ub:BB}
\Big( \frac{1}{\veps^2}\,\mc B\big(b_\veps\,\big|\,\oline b\big)\Big)_\veps\,\Subset\,L^\infty\big(\R_+;L^1(\Omega)\big)
\end{equation}
(see also Lemma 4.1 of \cite{F_PhysD} for more details).
Inserting \eqref{eq:constr-equal} into \eqref{eq:constr-limit} and using the above unform bound,
since $\psi$ is compactly supported in $\Omega$ we immediately get that
\begin{align*}
 \lim_{\veps\to0}\left(-\int^T_0\int_{\Omega}\frac{1}{2\,\veps}\,b_\veps^2\,\div\psi\,\dx\,\dd t\right)\,&=\,
 \lim_{\veps\to0}\left(-\,\oline b\int^T_0\int_{\Omega}\beta_\veps\,\div\psi\,\dx\,\dd t\right) \\
 &=\,-\,\oline b\int^T_0\int_{\Omega}\beta\,\div\psi\,\dx\,\dd t\,,
\end{align*}
where the last equality follows from \eqref{conv:b-r_weak}.

Similarly, from \eqref{eq:dens_dec-2} and \eqref{conv:u_weak-H} we infer that
\begin{align*}
\lim_{\veps\to0}\int^T_0\int_\Omega e_3\times\rho_\veps\,u_\veps\cdot\psi\,\dx\,\dd t\,=\,\oline\rho\int^T_0\int_\Omega e_3\times u\cdot\psi\,\dx\,\dd t\,.
\end{align*}

Putting these relations together, we thus deduce that, in the limit $\veps\to0$, the following relation must hold true in the sense
of distributions:
\begin{equation} \label{constr:beta-u}
 \oline b\,\nabla\beta\,+\,\k\,\oline\rho\,e_3\times u\,=\,0\,.
\end{equation}
We now split our discussions in the two cases $\k=0$ and $\k=1$.

\paragraph{The case $\k=0$.}
When $\k=0$, relation \eqref{constr:beta-u} simply yields a trivial constraint on the function $\beta$, which must be independent of the space variable:
\begin{equation} \label{eq:beta-triv}
 \mbox{ for }\ \k=0\,,\qquad\qquad \beta\,=\,\beta(t)\,.
\end{equation}

\paragraph{The case $\k=1$.}
In the case $\k=1$, instead, relation \eqref{constr:beta-u} implies a non-trivial relation between $\beta$ and $u$.
Here below, we resort to the notation introduced before the statement of Theorem \ref{th:lim_Coriols}:
for any vector $V = (V_1,V_2,V_3)\in\R^3$, we denote by $V_h = (V_1,V_2)$, and we define $\nabla_h$ and $\divh$ to be the usual
differential operators acting only on the variables $x_h=(x_1,x_2)$.

With this notation at hand, we start by noticing that \eqref{constr:beta-u} implies that
\[
 \beta = \beta(t,x_h)
\]
must be independent of the vertical variable. Thus, $u_h$ also satisfies $u_h = u_h(t,x_h)$. In addition, reading \eqref{constr:beta-u} on the first two components only,
one gets that
\[
 u_h\,=\,\frac{\oline b}{\oline\rho}\,\nabla^\perp_h\beta\,,
\]
where we have set $\nabla^\perp_h = (-\d_2,\d_1)$. Therefore, we infer that $\divh u_h = 0$, which implies, together with \eqref{constr:div}, that
$\d_3u_3=0$, whence $u_3 = u_3(t,x_h)$. Using the complete-slip boundary conditions \eqref{bc:compl-slip}, this finally implies
that $u_3\equiv0$.
All in all, we have thus proved that
\begin{align}
\label{constr:TP_k=1}
 \mbox{ for }\ \k=1\,,\qquad\qquad \beta\,=\,\beta(t,x_h)\qquad \mbox{ and }\qquad u\,=\,\big(u_h,0\big)\,, \ \mbox{ with } \ 
 u_h\,=\,\frac{\oline b}{\oline\rho} \,\nabla_h^\perp\beta\,.
\end{align}
We remark that this can be interpreted as a sort of mathematical description of the celebrated Taylor-Proudman theorem in geophysics
(see \tsl{e.g.} \cite{C-R, Ped}).

\subsection{A hidden compactness result} \label{ss:hidden}

In the proof of both Theorem \ref{th:limit} and Theorem \ref{th:lim_Coriols}, we need an extra information about a suitable combination
of $\beta$ and $r$. In fact, more than that, we need to understand the relation between the scalar functions $\beta_\veps$ and $r_\veps$
in the limit process $\veps\to0$. This ingredient, absolutely crucial for obtaining our results, is \emph{new} with respect to both the classical
incompressible limit and fast rotation limit, and their combination.

The extra information, common to both cases $\k=0$ and $\k=1$, will be obtained by a sort of \emph{compensated compactness} argument. By anticipating
what we will do in Sections \ref{s:conv-0} and \ref{s:conv-1} below, let us introduce the following definitions:
\[
 V_\veps\,:=\,\rho_\veps\,u_\veps\qquad\qquad \mbox{ and }\qquad\qquad W_\veps\,:=\,b_\veps\,u_\veps\,.
\]
Resorting to relations \eqref{eq:b_decomp} and \eqref{eq:dens_dec-2}, we see that we can reformulate the first and last equations
in system \eqref{eq:reduced-MHD} in the following form:
\begin{equation} \label{eq:b-r_waves}
\left\{ \begin{array}{l}
         \veps\,\d_tr_\veps\,+\,\div V_\veps\,=\,0 \\[1ex]
         \veps\,\d_t\beta_\veps\,+\,\div W_\veps\,=\,0\,.
        \end{array}
\right.
\end{equation}
We omit to reformulate the momentum equation, as (differently from what we will consider here) this depends on the choice of $\k\in\{0,1\}$.

Next, define the quantity
\[
 \alpha_\veps\,:=\,\oline\rho\,\beta_\veps\,-\,\oline b\,r_\veps\,.
\]
Observe that, owing to \eqref{eq:b_decomp} and \eqref{ub:r_eps}, we have
\begin{equation} \label{ub:alpha_e}
\big(\alpha_\veps\big)_\veps\,\Subset\,L^\infty\big(\R_+;L^m(\Omega)\big)\,,\qquad\qquad \mbox{ with }\quad m\,:=\,\min\{2,\g\}\,.
\end{equation}
At the same time, from the definitions of $V_\veps$ and $W_\veps$ and relations \eqref{eq:b_decomp} and \eqref{eq:dens_dec-2} again,
straightforward computations yield
\begin{align*}
 \oline\rho\,W_\veps\,-\,\oline b\,V_\veps\,&=\,\oline\rho\,\oline b\,u_\veps\,+\,\veps\,\oline\rho\,\beta_\veps\,u_\veps\,-\,\oline\rho\,\oline b\,u_\veps\,-\,
 \veps\,\oline b\,r_\veps\,u_\veps \\
 &=\, \veps\,\Big(\oline\rho\,\beta_\veps\,-\,\oline b\,r_\veps\Big)\,u_\veps \\
&=\,\veps\,\alpha_\veps\,u_\veps
\end{align*}
Therefore, using the previous relation, from system \eqref{eq:b-r_waves} we see that
\begin{equation} \label{eq:alpha}
 \d_t\alpha_\veps\,=\,\div\Big(\alpha_\veps\, u_\veps\Big)\,.
\end{equation}

At this point, taking advantage of \eqref{ub:alpha_e} and \eqref{ub:u_H^1}, along with the condition that $\g>3/2$, equation \eqref{eq:alpha} yields that
\[
\big(\d_t\alpha_\veps\big)_\veps\,\Subset\,L^2_\loc\big(\R_+;W^{-1,p}(\Omega)\big)\,,\qquad\qquad \mbox{ with }\quad
\frac{1}{p}\,:=\,\frac{1}{m}\,+\,\frac{1}{6}\,<\,\frac{5}{6}\,.
\]
By the Aubin-Lions lemma, thanks to the embedding $L^m(\Omega)\hookrightarrow L^p(\Omega)$,
the previous bound implies that the sequence of $\alpha_\veps$'s is compact in \tsl{e.g.} the space
$L^2_\loc\big(\R_+;W^{-1,p}(\Omega)\big)$, thus, up to a suitable extraction,
(by uniqueness of the weak limit) one gets the strong convergence
\begin{equation} \label{conv:alpha_strong}
 \alpha_\veps\,\longrightarrow\,\alpha\,:=\,\oline\rho\,\beta\,-\,\oline b\,r \qquad\qquad \mbox{ in }\qquad L^2_\loc\big(\R_+;W^{-1,p}(\Omega)\big)
\end{equation}
when $\veps\to0$, where $r$ and $\beta$ are precisely the profiles identified in \eqref{conv:b-r_weak}.

As a matter of fact, going back to \eqref{ub:alpha_e} and using an interpolation argument, we infer that the strong convergence holds true
in any intermediate space between $L^2_\loc\big(\R_+;L^{m}(\Omega)\big)$ and the space $L^2_\loc\big(\R_+;W^{-1,p}(\Omega)\big)$,
thus in the space $L^2_\loc\big(\R_+;L^p(\Omega)\big)$ (as $p<m$).
Now, taking advantage of the fact that $\big(u_\veps\big)_\veps\,\Subset\,L^2_\loc\big(\R_+;H^1(\Omega)\big)$ and is weakly convergent in that space,
it is possible to give sense to the product $\alpha_\veps\,u_\veps$ 
(recall that $p>6/5$ with strict inequality) and to apply a weak-convergence/strong-convergence argument to pass to the limit in it.
Omitting the details, which are somehow standard, one can in turn pass to the limit in the wek formulation
of equation \eqref{eq:alpha} and prove that the target $\alpha\,:=\,\oline\rho\,\beta\,-\,\oline b\,r$ solves, in the weak sense, the transport equation
\begin{equation} \label{eq:alpha-lim}
\d_t\alpha\,+\,u\cdot\nabla \alpha\,=\,0\,,\qquad\qquad \alpha_{|t=0}\,=\,\alpha_0\,:=\,\oline\rho\,\beta_0\,-\,\oline b\,r_0\,,
\end{equation}
where $\beta_0$ and $r_0$ have been identified in \eqref{conv:r-beta_init} and we have also used the divergence-free condition \eqref{constr:div}
satisfied by $u$.

As a last comment, we point out that the previous analysis works also in the case $\k=1$, in which we take $\Omega = \T^2\times\,]0,1[\,$, as in this case the problem
can be made periodic in the vertical variable (see the discussion in the Introduction and Subsection \ref{ss:k=1_convect} for more details).

\section{Proof of convergence: the case $\k=0$} \label{s:conv-0}
In this section, we complete the proof of Theorem \ref{th:limit}. Thus, throughout this part we fix the choice $\k=0$ in the momentum equation in \eqref{eq:reduced-MHD}.

We have already computed the limit of the equation for the magnetic function $b_\veps$ in Subsection \ref{ss:weak_conv}, keep in mind relation \eqref{constr:div}.
It goes without saying that an analogous argument applies also to the mass equation, bringing no additional difficulties, nor (unfortunately)
any additional information on the target dynamics.
Thus, in order to complete the proof of Theorem \ref{th:limit}, we only need to pass to the limit $\veps\to0$ in the momentum equation. For this, we use
its weak formulation \eqref{eq:u_weak}, but computed on a test-function $\psi$ such that
\begin{equation} \label{eq:k=0_test}
 \psi\in C^\infty_0\big(\R_+\times\Omega;\R^3\big)\,,\qquad\qquad \mbox{ with }\qquad \div\psi\,=\,0\,.
\end{equation}
Calling $T>0$ the time such that $\Supp\psi\subset[0,T]\times\Omega$ and using the divergence-free condition over $\psi$, relation
\eqref{eq:u_weak} with $\k=0$ simply becomes
\begin{align}
\label{eq:weak_incompr}
&-\int^T_0\int_{\Omega}\rho_\veps\,u_\veps\cdot\d_t\psi\,\dx\,\dd t\,-\,\int^T_0\int_{\Omega}\rho_\veps\,u_\veps\otimes u_\veps:\nabla\psi\,\dx\,\dd t \\
\nonumber
&\qquad\qquad\qquad\qquad +\int^T_0\int_{\Omega}\mbb S(\D u_\veps):\D\psi\,\dx\,\dd t\,=\,\int_\Omega\rho_{0,\veps}\,u_{0,\veps}\cdot\psi(0,\cdot)\,\dx\,.
\end{align}

Observe that the viscosity term is linear in $u_\veps$, hence computing its limit is straightforward:
\begin{equation*}
\lim_{\veps\to0}\int^T_0\int_{\Omega}\mbb S(\D u_\veps):\D\psi\,\dx\,\dd t\,=\,\int^T_0\int_{\Omega}\mbb S(\D u):\D\psi\,\dx\,\dd t\,,
\end{equation*}
where $u$ is the target velocity profile identified in \eqref{conv:u_weak-H}. In addition, owing to the assumptions on the initial data, see Subsection \ref{ss:data},
it is easy to check that
\begin{equation*}
\lim_{\veps\to0}\int_\Omega\rho_{0,\veps}\,u_{0,\veps}\cdot\psi(0,\cdot)\,\dx\,=\,\int_\Omega\oline\rho\,u_{0}\cdot\psi(0,\cdot)\,\dx\,,
\end{equation*}
wheree $u_0$ is the weak-limit of the sequence of initial velocities $\big(u_{0,\veps}\big)_\veps$.
Likewise, making use of \eqref{eq:dens_dec-2}, \eqref{ub:r_eps} and \eqref{conv:u_weak-H}, we can compute
\begin{equation*}
 \lim_{\veps\to0}\int^T_0\int_{\Omega}\rho_\veps\,u_\veps\cdot\d_t\psi\,\dx\,\dd t\,=\,
 \int^T_0\int_{\Omega}\oline\rho\,u\cdot\d_t\psi\,\dx\,\dd t\,.
\end{equation*}

Therefore, in order to complete the proof, we need to consider only the convective term, that is the last term appearing in the first line of \eqref{eq:weak_incompr}.
%
Our argument to compute its limit 
is based on a compensated compactness method. This, in turn, relies
on the study of the wave system, which governs the propagation of acoustic waves: let us start by presenting it.

\subsection{The system of acoustic waves} \label{ss:k=0_wave}

Recall the definition of $\beta_\veps$ and $r_\veps$ given in \eqref{eq:b_decomp} and \eqref{eq:dens_dec-2}, respectively. We further recall the following
notation from Subsection \ref{ss:hidden}:
\[
 V_\veps\,:=\,\rho_\veps\,u_\veps\qquad\qquad \mbox{ and }\qquad\qquad W_\veps\,:=\,b_\veps\,u_\veps\,.
\]
Resorting to the decomposition \eqref{eq:constr-equal} for the magnetic pressure term, we see that we can reformulate system \eqref{eq:reduced-MHD} in the following
equivalent form:
\begin{equation} \label{eq:k=0_waves}
\left\{ \begin{array}{l}
         \veps\,\d_tr_\veps\,+\,\div V_\veps\,=\,0 \\[1ex]
         \veps\,\d_t\beta_\veps\,+\,\div W_\veps\,=\,0 \\[1ex]
         \veps\,\d_tV_\veps\,+\,\oline b\,\nabla\beta_\veps\,=\,\veps\,F_\veps\,,
        \end{array}
\right.
\end{equation}
where we have defined
\[
 F_\veps\,:=\,-\,\div\big(\rho_\veps\,u_\veps\otimes u_\veps\big)\,+\,\div\mbb S(\D u_\veps)\,-\,\nabla P(\rho_\veps)\,-\,
 \frac{1}{\veps^2}\,\mc B\big(b_\veps\,\big|\,\oline b\big)\,.
\]

Before moving on, let us make a remark.
\begin{rem} \label{r:rho-small}
The wave system \eqref{eq:k=0_waves} containts three equations which are all coupled.  Despite various strategies are possible to treat algebraically the
convective term, precisely because of this coupling
in one way or another we always need the functions $\beta_\veps$ to speak with the density perturbation functions $r_\veps$.

This is the main reason for considering initial densities which are small perturbations of a constant state and for imposing assumption \eqref{est:b-r_in},
which in turn allows us to propagate that smallness and find suitable \emph{uniform} bounds for the $r_\veps$'s.
\end{rem}

Next, we establish suitable uniform bounds for the different quantities appearing in the wave system \eqref{eq:k=0_waves}.

First of all,
observe that, owing to the uniform bounds \eqref{est:u_en}, \eqref{eq:dens_decomp}, \eqref{ub:s_eps} and \eqref{ub:BB}, by
making use of the Sobolev embedding $L^1(\Omega)\hookrightarrow H^{-3/2-\delta}(\Omega)$ for $\delta>0$ arbitrarily small, we infer that
\begin{equation} \label{ub:k=0-f_e}
\forall\,T>0\,,\qquad\qquad 
 \big(F_\veps\big)_\veps\,\Subset\,L^2\big([0,T];H^{-2}(\Omega)\big)\,.
\end{equation}
As a matter of fact, in order to gather the $L^2$ integrability in time of the $F_\veps$'s, we have to decompose the convective term in the following way,
\[
 \rho_\veps\,u_{\veps}\otimes u_\veps\, = \,\sqrt{\rho_\veps}\;\sqrt{\rho_\veps}\,u_\veps\otimes u_\veps\,,
\]
and use the embedding $H^1(\Omega)\hookrightarrow L^6(\Omega)$ to bound the last factor in $L^2_T(L^6)$.

In addition, using \eqref{est:u_en}, \eqref{eq:dens_decomp} and \eqref{ub:s_eps} again, together with \eqref{eq:b_decomp}, we see that
there exists some $p_0>1$ such that
\begin{equation*}
\forall\,T>0\,,\qquad\qquad
 \big(V_\veps\big)_\veps\,,\, \big(W_\veps\big)_\veps\;\Subset\;L^2\big([0,T];L^{p_0}(\Omega)\big)\,,
\end{equation*}
where the precise value of $p_0$ depends of course on the quantity $m\,=\,\min\{2,\g\}$.
More precisely, we can decompose those velocity fields as follows
\begin{equation} \label{eq:V-W_decomp}
 V_\veps\,=\,\mbb V_\veps\,+\,\veps\,\mc V_\veps\qquad \mbox{ and }\qquad 
W_\veps\,=\,\mbb W_\veps\,+\,\veps\,\mc W_\veps\,,
\end{equation}
where we have defined
\[
 \mbb V_\veps\,:=\,\oline\rho\,u_\veps\,,\qquad \mbb W_\veps\,:=\,\oline b\,u_{\veps}\,,\qquad
 \mc V_\veps\,:=\,r_\veps\,u_\veps\,,\qquad \mc W_\veps\,:=\,\beta_\veps\,u_\veps\,.
\]
Thus, gathering the uniform boundedness properties of Subsections \ref{ss:bounds_energy} and \ref{ss:bounds_transport}, one sees that
\begin{align}
\label{ub:V-W_decomp}
&\big(\mbb V_\veps\big)_\veps\,,\,\big(\mbb W_\veps\big)_\veps\;\Subset\; L^2_\loc\big(\R_+;H^1(\Omega)\big)\,, \qquad\qquad
\big(\mc V_\veps\big)_\veps\,,\,\big(\mc W_\veps\big)_\veps\;\Subset\;L^2_\loc\big(\R_+;L^{p_0}(\Omega)\big)\,,
\end{align}
where $p_0>1$ is fixed as above.

\subsection{Regularisation} \label{ss:k=0_reg}
As a preparation to the compensated compactness argument, which requires to work on the very algebraic structure of the
convective term by exploiting the wave system \eqref{eq:k=0_waves}, we need to perform a regularisation procedure
on the various quantities under study.

Let us introduce a smoothing operator $\mc L_M$, depending on some large parameter $M\in\N$, such that
(formally speaking) $\mc L_M\longrightarrow\Id$ when $M\to+\infty$. 
Here, we understand that $\mc L_M$ only acts with respect to
the space variables and commutes with both time and space differentiation operators.
To fix ideas, we assume here that $\mc L_M$ is a a smooth truncation on the frequency space (as $\Omega=\T^3$ in our case),
for frequencies $k\in\Z^3$ which are such that $|k|\leq M$.
Then, the following properties can be established in a classical way. First of all, we have a trivial action on constant functions:
\[
\forall\,h\in\R\,,\qquad 
 \mc L_M h\,=\,c\,h\,,
\]
where the constant $c>0$ depends only on the integral of the convolution kernel associated to the frequency cut-off operator and is, in particular,
independent of $M\in\N$.
In addition, given a function $f\in W^{\ell,p}(\Omega)$, for some $\ell\in\Z$ and $p\in[1,+\infty]$, one has that
\begin{align*}
 \big(\mc L_Mf\big)_{M}\Subset W^{\ell,p}(\Omega)\,,\qquad\qquad &\mbox{ with }\qquad \forall\,M\in\N\,,\quad 
 \left\|\mc L_Mf\right\|_{W^{\ell,p}}\,\leq\,C\,\left\|f\right\|_{W^{\ell,p}} \\
 &\qquad\qquad \mbox{ and }\qquad \mc L_Mf\,\longrightarrow\,f\quad \mbox{ in }\ W^{\ell,p}(\Omega)\,,
\end{align*}
where the constant $C>0$ is independent of $M$, $\ell$, $p$ and $f$, and where we stress that the convergence is with respect to the strong topology of
$W^{\ell,p}$ when $M\to+\infty$.
Moreover, $\mc L_Mf$ is a smooth function: for any $M\in\N$ fixed, one has
\[
 \mc L_Mf\in \bigcap_{m\in\Z}W^{m,p}\,,\qquad\qquad \mbox{ with }\qquad \forall\,m\geq\ell\,,\quad 
 \left\|\mc L_Mf\right\|_{W^{m,p}}\,\leq\,C\,2^{M(m-\ell)}\,\left\|f\right\|_{W^{\ell,p}}\,,
\]
where the multiplicative constant $C>0$ satisfies the same property as above.
Finally, let us mention that, given any $p\in[1,+\infty]$ and any $f\in W^{1,p}(\Omega)$, one has
\[
\left\|\big(\Id - \mc L_M\big)f\right\|_{L^p}\,\leq\,C\,2^{-M}\,\left\|\nabla f\right\|_{L^p}\,,
\]
where the constant $C$ is independent of $M$, $f$ and $p$.

For convenience of notation, whenever we deal with a sequence of functions $\big(f_\veps\big)_\veps$, we define
\[
f_{\veps,M}\,:=\,\mc L_Mf_\veps\,,
\]

Let us now apply the regularising operators $\mc L_M$, for large $M\in\N$, to the unknowns of our system \eqref{eq:reduced-MHD}.
Consider $u_{\veps,M}\,=\,\mc L_Mu_\veps$. By the properties listed above for the operators $\mc L_M$, combined with the uniform bound
\eqref{ub:u_H^1}, we deduce
\begin{equation} \label{est:u-e-M}
\left\|u_{\veps,M}\right\|_{L^2_T(H^s)}\,\leq\,C(T,s,M)\,,
\end{equation}
for a constant depending only on the quantities in the brackets, but not on $\veps\in\,]0,1]$.

Moreover, using the uniform bounds \eqref{ub:u_H^1} and the above properties for the operator $\mc L_M$, we immediately infer that,
for any $T>0$ fixed, one has
\begin{equation} \label{est:Id-S_M-u}
\left\|\big(\Id-\mc L_M\big)u_\veps\right\|_{L^2_T(L^2)}\,\leq\,C\,2^{-M}\,\,,
\end{equation}
for a constant $C>0$ independent of $\veps$.

In addition for later use, let us consider the target velocity profile identified in \eqref{conv:u_weak-H}. Observe that, owing to
\eqref{conv:u_weak-H} and the continuity properties of the regularising operators $\mc L_M$, we get that
\begin{equation} \label{conv:u_e-M_to_u_M}
\forall\,T>0\,,\ \forall\,M>0\,,\qquad\qquad u_{\veps,M}\,\rightharpoonup\,u_M\qquad \mbox{ in }\quad L^2\big([0,T];H^1(\Omega)\big)
\qquad \mbox{ for }\ \veps\to0\,.
\end{equation}
Moreover, thanks to the properties of the $\mc L_M$ operators and Lebesgue dominated convergence theorem, we have the strong convergence
\begin{equation} \label{conv:strong_u}
u_M\,=\,\mc L_Mu\,\longrightarrow\,u \qquad\qquad \mbox{ in }\qquad L^2\big([0,T];H^1\big)\,,\qquad \mbox{ for }\quad M\to+\infty\,.
\end{equation}
This argument shows that, if we can pass to the limit in the convective term, where we have regularised
the velocity fields, then we can easily compute the limit when the appoximation parameter $M$ goes to $+\infty$.

Finally, let us apply the regularising operators $\mc L_M$ to the wave system \eqref{eq:k=0_waves}: with the notation introduced above, since the system
has constant coefficients, we get
\begin{equation} \label{eq:k=0_waves_M}
\left\{ \begin{array}{l}
         \veps\,\d_tr_{\veps,M}\,+\,\div V_{\veps,M}\,=\,0 \\[1ex]
         \veps\,\d_t\beta_{\veps,M}\,+\,\div W_{\veps,M}\,=\,0 \\[1ex]
         \veps\,\d_tV_{\veps,M}\,+\,\oline b\,\nabla\beta_{\veps,M}\,=\,\veps\,F_{\veps,M}\,.
        \end{array}
\right.
\end{equation}
Recall that all the quantities appearing in the above system are smooth with respect to the space variable. In particular, for any $M\in\N$,
any $T>0$ and any $s>0$, we have
\begin{equation} \label{ub:F_e-M}
 \sup_{\veps\in\,]0,1]}\left\|\left(r_{\veps,M}\,,\,\beta_{\veps,M}\right)\right\|_{L^\infty_T(H^s)}\,+\,
 \sup_{\veps\in\,]0,1]}\left\|F_{\veps,M}\right\|_{L^2_T(H^s)}\,\leq\,C(s,M,T)\,,
\end{equation}
for a constant $C(s,M,T)>0$ only depending on the quantities inside the brackets. In addition, keeping in mind \eqref{eq:V-W_decomp} and \eqref{ub:V-W_decomp},
we can write
\begin{equation} \label{eq:V-W_M}
 V_{\veps,M}\,=\,\mbb V_{\veps,M}\,+\,\veps\,\mc V_{\veps,M} \qquad \mbox{ and }\qquad
W_{\veps,M}\,=\,\mbb W_{\veps,M}\,+\,\veps\,\mc W_{\veps,M}\,,
\end{equation}
where we simply have
\[
 \mbb V_{\veps,M}\,=\,\oline\rho\,u_{\veps,M} \qquad \mbox{ and }\qquad \mbb W_{\veps,M}\,=\,\oline b\,u_{\veps,M}\,,
\]
together with the properties
\begin{equation} \label{ub:vf_e-M}
  \sup_{\veps\in\,]0,1]}\left\|\left(\mc V_{\veps,M}\,,\,\mc W_{\veps,M}\right)\right\|_{L^2_T(H^s)}\,\leq\,C(s,M,T)\,.
\end{equation}

\subsection{Approximation of the convective term} \label{ss:k=0_approx}

After developing the previous preparatory material, we can tackle the proof of the convergence of the convective term, namely of the quantity
\[
\int^T_0\int_{\Omega}\rho_\veps\,u_\veps\otimes u_\veps:\nabla\psi\,\dx\,\dd t\,,
\]
where the test-function $\psi$ is as in \eqref{eq:k=0_test}.

Some approximation is however still needed. To begin with, we observe that, thanks to the bound established in \eqref{est:Id-S_M-u},
we have the equality
\begin{align*}
 \lim_{\veps\to0}\int^T_0\int_{\Omega}\rho_\veps\,u_\veps\otimes u_\veps:\nabla\psi\,\dx\,\dd t\,&=\,
 \lim_{M\to+\infty}\limsup_{\veps\to0}\int^T_0\int_{\Omega}\rho_\veps\,u_\veps\otimes u_{\veps,M}:\nabla\psi\,\dx\,\dd t \\
 &=\,\lim_{M\to+\infty}\limsup_{\veps\to0}\frac{1}{\oline\rho}\int^T_0\int_{\Omega}\rho_\veps\,u_\veps\otimes \mbb V_{\veps,M}:\nabla\psi\,\dx\,\dd t
 \,. 
\end{align*}
Using formula \eqref{eq:V-W_M} for $V_{\veps,M}$, asince now this quantity is smooth and as much integrable as we want in space, we immediately infer that
\begin{align*}
 \lim_{\veps\to0}\int^T_0\int_{\Omega}\rho_\veps\,u_\veps\otimes u_\veps:\nabla\psi\,\dx\,\dd t\,&=\,
\lim_{M\to+\infty}\limsup_{\veps\to0}\frac{1}{\oline\rho}\int^T_0\int_{\Omega} V_\veps \otimes V_{\veps,M}:\nabla\psi\,\dx\,\dd t
 \,,
\end{align*}
where we recall that $V_\veps\,:=\,\rho_\veps\,u_\veps$. Now, we can use the following decomposition, which follows from the definitions:
\begin{align*}
V_{\veps}\,=\,\mbb V_{\veps}\,+\,\veps \mc V_{\veps}\,=\,\mbb V_{\veps,M}\,+\,\oline\rho\,\big(\Id-\mc L_M\big)u_\veps\,+\,\veps\,\mc V_{\veps}\,.
\end{align*}
Owing to properties \eqref{ub:V-W_decomp} and \eqref{est:Id-S_M-u}, we then deduce that 
\begin{align*}
 \lim_{\veps\to0}\int^T_0\int_{\Omega}\rho_\veps\,u_\veps\otimes u_\veps:\nabla\psi\,\dx\,\dd t\,&=\,
\lim_{M\to+\infty}\limsup_{\veps\to0}\frac{1}{\oline\rho}\int^T_0\int_{\Omega} \mbb V_{\veps,M} \otimes V_{\veps,M}:\nabla\psi\,\dx\,\dd t \,.
\end{align*}
Finally, a further application of the decomposition \eqref{eq:V-W_M} allows to find that
\begin{align} \label{eq:final_convect}
 \lim_{\veps\to0}\int^T_0\int_{\Omega}\rho_\veps\,u_\veps\otimes u_\veps:\nabla\psi\,\dx\,\dd t\,&=\,
\lim_{M\to+\infty}\limsup_{\veps\to0}\frac{1}{\oline\rho}\int^T_0\int_{\Omega} V_{\veps,M} \otimes V_{\veps,M}:\nabla\psi\,\dx\,\dd t \,.
\end{align}

Thus, from now on we will focus on computing
\begin{equation} \label{eq:k=0_to-prove}
 \limsup_{\veps\to0}\frac{1}{\oline\rho}\int^T_0\int_{\Omega} V_{\veps,M} \otimes V_{\veps,M}:\nabla\psi\,\dx\,\dd t\,,
\end{equation}
up to terms which must vanish in the subsequent limit $M\to+\infty$.

\subsection{Computing the limit \tsl{via} compensated compactness} \label{ss:k=0_comp-comp}

In light of the previous computations, we only have to consider the convergence of the integral appearing in \eqref{eq:k=0_to-prove}.
In order to do so, let us introduce the Leray-Helmholtz decomposition over $\Omega$: we denote
by $\P$ and $\Q$ the $L^2$-orthogonal projections onto, respectively, the space of divergence-free vector fields and its orthogonal complement.
Recall that, for any vector field $H\in L^2(\Omega;\R^3)$, one has $\div \P H = 0$ and
$\Q H = \nabla\Phi$, for a suitable function $\Phi\in \dot H^1(\Omega)$ such that 
$-\Delta\Phi = -\div H$.

Now, let us apply the operator $\P$ to the regularised wave system \eqref{eq:k=0_waves_M}, and more precisely to the last equation therein.
Since $\P\nabla\beta_{\veps,M} = 0$, we get
\[
 \d_t\P V_{\veps,M}\,=\,\P F_{\veps,M}\,.
\]
Thus, owing to the first equality in \eqref{eq:V-W_M} and to the boundedness properties \eqref{est:u-e-M}, \eqref{ub:F_e-M} and \eqref{ub:vf_e-M}, 
an application of the Aubin-Lions lemma yields that, for any \emph{fixed} $M\in\N$, the sequence (in $\veps$!)
$\big(\P V_{\veps,M}\big)_\veps$ is compact in $L^2\big([0,T]\times\Omega\big)$, for any $T>0$ fixed.
Thanks to \eqref{eq:V-W_M} again together with \eqref{conv:u_e-M_to_u_M}, we infer that
\begin{equation} \label{conv:strong_PV}
 \forall\,T>0\,,\ \forall\,M\in\N\,,\qquad\qquad \P V_{\veps,M}\,\longrightarrow\,\oline\rho\,\P u_M = \oline\rho\,u_M\qquad \mbox{ in }\quad L^2\big([0,T]\times\Omega\big)
\end{equation}
in the limit $\veps\to0$, where we have also used the constraint \eqref{constr:div} to write $\P u_M = u_M$.
Of course, \eqref{constr:div} also implies, together with \eqref{conv:u_e-M_to_u_M} and \eqref{eq:V-W_M}, the weak convergence
\begin{equation} \label{conv:QV}
 \Q V_{\veps,M}\,\rightharpoonup\,0\qquad\qquad \mbox{ in }\qquad L^2\big([0,T];H^1(\Omega)\big)
\end{equation}
in the limit $\veps\to0$, for any value of $M\in\N$ fixed.

Let us assume the next statement for a while.
\begin{lemma} \label{l:QxQ}
 Let $\psi$ be a test-function as in \eqref{eq:k=0_test}, with $\Supp\psi\subset[0,T]\times\Omega$, for some $T>0$. Then we have
\[
 \lim_{M\to+\infty}\limsup_{\veps\to0}\frac{1}{\oline\rho}\int^T_0\int_{\Omega} \Q V_{\veps,M} \otimes \Q V_{\veps,M}:\nabla\psi\,\dx\,\dd t\,=\,0
\]
\end{lemma}

Then, we can complete the proof of the convergence of the integral in \eqref{eq:final_convect}. Indeed, using the decomposition
$V_{\veps,M}\,=\,\P V_{\veps,M} + \Q V_{\veps,M}$ together with
\eqref{conv:strong_PV}, \eqref{conv:QV} and \eqref{conv:strong_u}, the previous lemma used implies that
\begin{align*}
 \lim_{\veps\to0}\int^T_0\int_{\Omega}\rho_\veps\,u_\veps\otimes u_\veps:\nabla\psi\,\dx\,\dd t\,&=\,
\lim_{M\to+\infty}\limsup_{\veps\to0}\frac{1}{\oline\rho}\int^T_0\int_{\Omega} V_{\veps,M} \otimes V_{\veps,M}:\nabla\psi\,\dx\,\dd t \\
&=\,\int^T_0\int_\Omega \oline\rho\,u\otimes u:\nabla\psi\,\dx\,\dd t\,,
\end{align*}
thus completing the proof of Theorem \ref{th:limit}.

Therefore, let us now tackle the proof of Lemma \ref{l:QxQ}.
In our argument, we will use the following convenient notation:
we will denote by $\mc R_{\veps,M}$  any remainder term, i.e. any term having the property that
\begin{equation} \label{eq:remainder}
\lim_{M\ra+\infty}\,\limsup_{\veps\ra0}\,\left|\int^T_0\int_\Omega \mc{R}_{\veps,M}\,\cdot\,\psi\,\dx\,\dd t\right|\,=\,0
\end{equation}
for all test functions $\psi$ as in \eqref{eq:k=0_test}.

\begin{proof}[Proof of Lemma \ref{l:QxQ}]
Since all the quantities appearing in the integral are smooth in space, we can perform an integration by parts and reduce our study to the convergence of the term
$\div\big(\Q V_{\veps,M}\otimes\Q V_{\veps,M}\big)$. Direct algebraic manipulations give us
\begin{align*}
\div\big(\Q V_{\veps,M}\otimes \Q V_{\veps,M}\big)\,&=\,\div\big( \Q V_{\veps,M}\big)\;\Q V_{\veps,M}\,+\,\big(\Q V_{\veps,M}\cdot\nabla\big)\Q V_{\veps,M} \\
&=\,\div\big( \Q V_{\veps,M}\big)\;\Q V_{\veps,M}\,+\,\frac{1}{2}\nabla\big|\Q V_{\veps,M}|^2\,+\,\curl\big(\Q V_{\veps,M}\big)\,\times\,\Q V_{\veps,M}\,.
\end{align*}
Observe that the gradient term identically vanishes, when tested against a test-function $\psi$ satisfying the properties in \eqref{eq:k=0_test}.
At the same time, by definition of the projection operator $\Q$ one has
\[
 \div\big(\Q V_{\veps,M}\big)\,=\,\div\big( V_{\veps,M}\big) \qquad \mbox{ and }\qquad \curl\big(\Q V_{\veps,M}\big)\,=\,0\,.
\]
Thus, resorting to the notation introduced in \eqref{eq:remainder}, the above equality becomes
\begin{align*}
\div\big(\Q V_{\veps,M}\otimes \Q V_{\veps,M}\big)\,&=\,\div\big( V_{\veps,M}\big)\;\Q V_{\veps,M}\,+\,\mc R_{\veps,M}\,.
\end{align*}

At this point, we make use of the wave system \eqref{eq:k=0_waves_M} to further write, in the sense of distributions, the equality
\begin{align*}
\div\big(\Q V_{\veps,M}\otimes \Q V_{\veps,M}\big)\,&=\,-\,\veps\,\d_tr_{\veps,M}\;\Q V_{\veps,M}\,+\,\mc R_{\veps,M} \\
&=\,r_{\veps,M}\,\veps\,\d_t\Q V_{\veps,M}\,+\,\mc R_{\veps,M} \\
&=\,\veps\,r_{\veps,M}\,\Q F_{\veps,M}\,-\,\oline b\,r_{\veps,M}\,\nabla\beta_{\veps,M}\,+\,\mc R_{\veps,M} \\
&=\,-\,\oline b\,r_{\veps,M}\,\nabla\beta_{\veps,M}\,+\,\mc R_{\veps,M}\,.
\end{align*}
Observe that the term presenting the total time derivative has been treated as a remainder term, in the sense of \eqref{eq:remainder}. Indeed, the time derivative
can be put on the test-function and the factor $\veps$ in front of the integral makes everything vanish in the limit $\veps\to0$.

As a last step, we introduce the function
\[
 \alpha_{\veps,M}\,:=\,\oline\rho\,\beta_{\veps,M}\,-\,\oline b\,r_{\veps,M}\,=\,\mc L_M(\alpha_\veps)\,,
\]
where $\alpha_\veps$ is precisely the quantity defined in Subsection \ref{ss:hidden}. Then, we can write
\[
 -\,\oline b\,r_{\veps,M}\,\nabla\beta_{\veps,M}\,=\,\alpha_{\veps,M}\,\nabla\beta_{\veps,M}\,-\,\frac{\oline\rho}{2}\,\nabla\beta_{\veps,M}^2\,.
\]
Insterting this equality in the previous one, and moving the pure gradient term into the remainder $\mc R_{\veps,M}$ (as one has $\div\psi=0$),
we find
\begin{align*}
\div\big(\Q V_{\veps,M}\otimes \Q V_{\veps,M}\big)\,&=\,\alpha_{\veps,M}\,\nabla\beta_{\veps,M}\,+\,\mc R_{\veps,M}\,.
\end{align*}
The first term on the right-hand side is still non-linear, so some compactness in space-time is needed on one of the two factors to pass to the limit in it.
Observe that, repeating \tsl{mutatis mutandis} the computations of Subsection \ref{ss:hidden}, from equations \eqref{eq:k=0_waves_M}
and \eqref{eq:V-W_M} we get that
\[
 \d_t\alpha_{\veps,M}\,=\,\div\Big(\oline\rho\,\mc W_{\veps,M}\,-\,\oline b\,\mc V_{\veps,M}\Big)\,,
\]
which implies that, for any $M\in\N$ \emph{fixed}, the sequence $\big(\alpha_{\veps,M}\big)_\veps$ is compact in \tsl{e.g.} the space
$L^2\big([0,T]\times \Omega\big)$. Therefore, one infers that, for any $M\in\N$ fixed, one has
\[
\alpha_{\veps,M}\,\nabla\beta_{\veps,M}\,{\rightharpoonup}\,\alpha_M\,\nabla\beta_M\qquad\qquad \mbox{ in }\qquad L^2\big([0,T]\times\Omega\big)
\]
in the limit $\veps\to0$. By linearity of the smoothing operators $\mc L_M$, however, $\beta_M=\mc L_M\beta$, with $\beta$ satisfying
\eqref{eq:beta-triv}. Hence, $\beta=\beta(t)$ is independent of the space variable, so
$\mc L_M\beta(t)\,=\,c\,\beta(t)$, where the constant $c>0$ depends only on the integral of the convolution kernel associated to the smoothing operator $\mc L_M$.
This implies that $\nabla \beta_M=0$, thus yielding that
\[
 \alpha_{\veps,M}\,\nabla\beta_{\veps,M}\,=\,\mc R_{\veps,M}\qquad\qquad \mbox{ in the sense of relation \eqref{eq:remainder}.}
\]
With this property at hand, the proof of Lemma \ref{l:QxQ} is completed.
\end{proof}

\section{Convergence in presence of fast rotation: the case $\k=1$} \label{s:conv-1}
In this section, we consider the case $\k=1$ and complete the proof of Theorem \ref{th:lim_Coriols}.
Recall that, in this framework, the space domain is $\Omega\,=\,\T^2\times\,]0,1[\,$ and we assume the complete-slip boundary conditions
\eqref{bc:compl-slip}. At the same time, as mentioned in the Introduction (see also the discussion in Subsection \ref{ss:k=1_convect} below), we recall that
the problem \eqref{eq:reduced-MHD}-\eqref{bc:compl-slip} can be transformed into a periodic problem also with respect to the third variable.

Observe that the velocity field now satisfies the constraints \eqref{constr:TP_k=1}. Therefore, we need to pass to the limit
in the weak formulation \eqref{eq:u_weak} of the momentum equation by using test-functions $\psi$ which, this time, satisfy
\begin{equation} \label{eq:k=1_test}
 \psi\,=\,\big(\nabla^\perp_h\vphi(t,x_h),0\big)\,,\qquad\qquad \mbox{ with }\qquad \vphi\in C^\infty_0\big(\R_+\times\T^2\big)\,.
\end{equation}
Notice that, in particular, one still has that $\div \psi=0$.
Therefore, after denoting $T>0$ the time such that $\Supp\vphi\subset[0,T]\times\T^2$, equation \eqref{eq:u_weak} tested on such a $\psi$ becomes
\begin{align}
\label{eq:weak_rotat}
&-\int^T_0\int_{\Omega}\rho_\veps\,u_\veps\cdot\d_t\psi\,\dx\,\dd t\,-\,\int^T_0\int_{\Omega}\rho_\veps\,u_\veps\otimes u_\veps:\nabla\psi\,\dx\,\dd t \\
\nonumber
&\qquad\quad 
\frac{1}{\veps}\int^T_0\int_\Omega e_3\times\rho_\veps\,u_\veps\cdot\psi\,\dx\,\dd t\,+\,
\int^T_0\int_{\Omega}\mbb S(\D u_\veps):\D\psi\,\dx\,\dd t\,=\,\int_\Omega\rho_{0,\veps}\,u_{0,\veps}\cdot\psi(0,\cdot)\,\dx\,.
\end{align}

Before proceeding further, let us recall the following notation from Subsection \ref{ss:result}: for any function $f$ defined on $\Omega$, we denote
by $\lan f\ran$ its vertical average, namely the quantity
\[
 \lan f\ran = \lan f\ran (t,x_h)\,:=\,\int_0^1f(t,x_h,z)\,\dd z\,.
\]
With this notation, we can start simplifying some terms in \eqref{eq:weak_rotat} and compute their limit for $\veps\to0$.

\subsection{Preliminary convergence results} \label{ss:k=1_first-converg}

By using the fact that $\psi$ in \eqref{eq:k=1_test} is horizontal (that is, $\psi_3\equiv0$) and only depends on the horizontal variables, we can rewrite
the term involving the initial datum as
\[
 \int_\Omega\rho_{0,\veps}\,u_{0,\veps}\cdot\psi(0,\cdot)\,\dx\,=\,\int_{T^2}\lan\rho_{0,\veps}\,\big(u_{0,\veps}\big)_h\ran\cdot\nabla_h^\perp\vphi(0,\cdot)\,\dd x_h\,.
\]
Using the assumptions fixed in Subsection \ref{ss:data}, it is then easy to compute
\begin{align*}
 \lim_{\veps\to0}\int_\Omega\rho_{0,\veps}\,u_{0,\veps}\cdot\psi(0,\cdot)\,\dx\,&=\,\int_{T^2}\oline\rho\, \lan u_{0,h}\ran\cdot\nabla_h^\perp\vphi(0,\cdot)\,\dd x_h \\
 &=\,-\int_{\T^2}\oline\rho\,\lan \o_0\ran\, \vphi(0,\cdot)\,\dd x_h\,,
\end{align*}
where we have defined $\o_0\,:=\,\curlh(u_{0,h})\,=\,\d_1u_{0,2}\,-\,\d_2u_{0,1}$.
The last integration by parts in the above equation is in fact not needed at this stage, but it will turn out to be useful in a while.

Similarly, owing to \eqref{eq:dens_dec-2} and \eqref{conv:u_weak-H}, we can pass to the limit in the $\d_t$-term.
As a matter of fact, using the properties appearing in
\eqref{constr:TP_k=1}, and precisely that $u_h = \frac{\oline b}{\oline\rho} \nabla_h^\perp\beta$ does not depend on $x_3$, we obtain the following series of equalities:
\begin{align*}
 \lim_{\veps\to0}\left(-\int^T_0\int_{\Omega}\rho_\veps\,u_\veps\cdot\d_t\psi\,\dx\,\dd t\right)\,&=\,
\lim_{\veps\to0}\left(-\int^T_0\int_{\T^2}\lan\rho_\veps\,u_{\veps,h}\ran \cdot\d_t\nabla_h^\perp\vphi\,\dd x_h\,\dd t\right) \\
&=\,-\,\oline\rho\int^T_0\int_{\T^2}u_{h} \cdot\d_t\nabla_h^\perp\vphi\,\dd x_h\,\dd t \\
&=\,-\,\oline b\int^T_0\int_{\T^2}\nabla_h\beta \cdot\d_t\nabla_h\vphi\,\dd x_h\,\dd t \\
&=\,\oline b\int^T_0\int_{\T^2}\Delta_h\beta \cdot\d_t\vphi\,\dd x_h\,\dd t \,,
\end{align*}
where we have also exploited the fact that, for any couple of vectors $A$ and $B$ of $\R^2$, one has $A^\perp\cdot B^\perp = A\cdot B$.
Observe that, with an abuse of notation, in the last line we have denoted as an integral a duality pair in the sense of distributions.

As for the viscosity term, we use the facts that 
$\psi_3=0$ and that $\psi$ does not depend on the vertical variable to write
\begin{align*}
 \int^T_0\int_{\Omega}\mbb S(\D u_\veps):\D\psi\,\dx\,\dd t \,&=\, \int^T_0\int_{\Omega}\D u_\veps:\D\psi\,\dx\,\dd t \\
 &=\,\int^T_0\int_{\T^2}\lan\D_h u_{\veps,h}\ran :\D_h\psi_h\,\dd x_h\,\dd t\,,
\end{align*}
where $\D_hM$ denotes the symmetric part of the Jacobian of a vector field $M$ only with respect to the horizontal derivatives $\d_1$ and $\d_2$.
Taking advantage of \eqref{conv:u_weak-H} again, we then find
\begin{align*}
\lim_{\veps\to0} \int^T_0\int_{\Omega}\mbb S(\D u_\veps):\D\psi\,\dx\,\dd t \,&=\,\int^T_0\int_{\T^2}\D_h u_{h} :\D_h\psi_h\,\dd x_h\,\dd t \\
&=\,\int^T_0\int_{\T^2}\Delta_h \o\, \vphi\,\dd x_h\,\dd t\,,
\end{align*}
where in the last line we have integrated by parts twice (as $\psi_h =\nabla_h^\perp\vphi$) and denoted $\o\,=\,\curlh u_h = \d_1u_2\,-\,\d_2u_1$.
Again, we have adopted an abuse of notation by writing an integral instead of a duality pairing in the sense of distributions.
Notice that $\o = \frac{\oline b}{\oline\rho}\Delta_h\beta$, in fact.

Therefore, to conclude the proof of Theorem \ref{th:lim_Coriols}, it remains only to deal with the Coriolis term and the convective term in \eqref{eq:weak_rotat}:
this is the goal of the next two subsections.

\subsection{The limit of the rotation term} \label{ss:k=1_rot}
We consider here the convergence of the Coriolis term, namely of the integral
\[
 \frac{1}{\veps}\int^T_0\int_\Omega e_3\times\rho_\veps\,u_\veps\cdot\psi\,\dx\,\dd t\,,
\]
where $\psi$ satisfies \eqref{eq:k=1_test}. In fact, exploiting the properties of such $\psi$ and the definition of the rotation operator, we can write
\begin{align*}
 \frac{1}{\veps}\int^T_0\int_\Omega e_3\times\rho_\veps\,u_\veps\cdot\psi\,\dx\,\dd t\,&=\,
 \frac{1}{\veps}\int^T_0\int_\Omega \rho_\veps\,u_{\veps,h}^\perp\cdot\psi_h\,\dx\,\dd t \\
 &=\,\frac{1}{\veps}\int^T_0\int_{\T^2} \lan\rho_\veps\,u_{\veps,h}^\perp\ran\cdot\nabla^\perp_h\vphi\,\dd x_h\,\dd t \\
 &=\,\frac{1}{\veps}\int^T_0\int_{\T^2} \lan\rho_\veps\,u_{\veps,h}\ran\cdot\nabla_h\vphi\,\dd x_h\,\dd t\,.
\end{align*}

At this point, we consider the mass equation, namely the first equation appearing in system \eqref{eq:reduced-MHD}. Testing it
against the chosen function $\vphi$, using that $\vphi$ does not depend on $x_3$, we get
\begin{align*}
-\int^T_0\int_{\T^2}\lan\rho_\veps-\oline\rho\ran\,\d_t\vphi\,\dd x_h\,\dd t\,-\,\int^T_0\int_{\T^2}\lan \rho_\veps\,u_{\veps,h}\ran\cdot\nabla_h\vphi\,\dd x_h\, \dd t\,=\,
\int_{\T^2}\lan\rho_{0,\veps}-\oline\rho\ran\,\vphi(0,\cdot)\,\dd x_h\,.
\end{align*}
In light of this equation, using relation \eqref{eq:dens_dec-2} and the same decomposition for the initial datum, we can remove the singularity
of the Coriolis term by writing
\begin{align*}
 \frac{1}{\veps}\int^T_0\int_\Omega e_3\times\rho_\veps\,u_\veps\cdot\psi\,\dx\,\dd t\,&=\,
 \frac{1}{\veps}\int^T_0\int_{\T^2} \lan\rho_\veps\,u_{\veps,h}\ran\cdot\nabla_h\vphi\,\dd x_h\,\dd t \\
 &=\,-\int^T_0\int_{\T^2}\lan r_\veps\ran\,\d_t\vphi\,\dd x_h\,\dd t\,-\,
 \int_{\T^2}\lan r_{0,\veps}\ran\,\vphi(0,\cdot)\,\dd x_h\,.
\end{align*}
We can thus compute its limit, finding
\begin{align*}
\lim_{\veps\to0} \frac{1}{\veps}\int^T_0\int_\Omega e_3\times\rho_\veps\,u_\veps\cdot\psi\,\dx\,\dd t\,&=\,
-\int^T_0\int_{\T^2}\lan r\ran\,\d_t\vphi\,\dd x_h\,\dd t\,-\, \int_{\T^2}\lan r_{0}\ran\,\vphi(0,\cdot)\,\dd x_h\,,
\end{align*}
where $r_0$ has been introduced in \eqref{conv:r-beta_init} and $r$ has been identified in \eqref{conv:b-r_weak}.

Before moving on, we formulate a couple of important observations.

\begin{rem} \label{r:vphi}
Our treatement of the Coriolis term somehow justifies the fact that, in the computations of Subsection \ref{ss:k=1_first-converg}, we have written
all the terms as scalar quantities tested against the scalar test-function $\vphi$, instead of keeping the vectorial equation computed on the vectorial
test-function $\psi$.
\end{rem}

\begin{rem} \label{r:r-beta}
While the terms of Subsection \ref{ss:k=1_first-converg} are written in terms of $u_h$, thus of $\beta$, the limit of the Coriolis term
depend on the target function $r$. Thus, in order to find a closed system, \emph{differently} from what usually happens in the low Mach number limit,
or in the low Rossby number limit, one needs to find an additional relation linking $r$ and $\beta$, which is exactly provided
by the equation found in Subsection \ref{ss:hidden}, see \eqref{conv:alpha_strong} and \eqref{eq:alpha-lim} above.
\end{rem}

\subsection{Convergence of the convective term} \label{ss:k=1_convect}

To conclude the proof of Theorem \ref{th:lim_Coriols}, it remains to deal once again, as in Section \ref{s:conv-0}, with the convective term. 
The technique is similar to the one employed for the case $\k=0$; however, since the domain
$\Omega$ is different here, we need to perform some preliminary reductions.

More precisely, we notice that the problem \eqref{eq:reduced-MHD} on $\Omega = \T^2\times\,]0,1[\,$ with complete-slip boundary conditions \eqref{bc:compl-slip}
can be recasted (see \cite{Ebin}) as a periodic problem also with respect to the vertical variable, in the new domain
$$
\wtilde\Omega\,=\,\T^2\,\times\,\wtilde{\mbb{T}}\,,\qquad\qquad \mbox{ with }\qquad \wtilde{\mbb{T}}\,:=\,[-1,1]/\sim\,,
$$
where $\sim$ denotes the equivalence relation which identifies $-1$ and $1$. Indeed, the equations remain invariant if we extend
$\rho$, $b$ and $u_h = (u_1,u_2)$ as even functions with respect to $x_3$, and $u_3$ as an odd function.
Therefore, from now on we will consider the equations \eqref{eq:reduced-MHD} set on the new space domain $\wtilde\Omega$.

Now that the problem has become periodic on all the variables, we can easily seen that the analysis of Subsections \ref{ss:k=0_reg}
and \ref{ss:k=0_approx}, based on regularising the velocity fields and approximating the convective term by the same term, but computed on the regularised
vector fields, can be applied also in this context $\k=1$. Therefore, without repeating all the passages here, from now on we will assume
that all the quantities we are dealing with are smooth with respect to the space variable, and our goal becomes to compute the limit
\begin{align}
\label{eq:k=1_lim-conv}
 \lim_{\veps\to0}\int^T_0\int_{\Omega}\rho_\veps\,u_\veps\otimes u_\veps:\nabla\psi\,\dx\,\dd t\,&=\,
 \lim_{\veps\to0}\frac{1}{\oline\rho}\int^T_0\int_{\Omega}V_\veps\otimes V_\veps:\nabla\psi\,\dx\,\dd t \\
\nonumber
 &=\, \lim_{\veps\to0}\frac{1}{\oline\rho}\int^T_0\int_{\T^2}\lan V_{\veps,h}\otimes V_{\veps,h}\ran :\nabla_h\nabla^\perp_h\vphi\,\dd x_h\,\dd t \,,
\end{align}
where we have used that the test-function $\psi$ satisfies the conditions in \eqref{eq:k=1_test} and we have resorted to the notation
$V_\veps\,=\,\rho_\veps\,u_\veps$ from Subsection \ref{ss:k=0_wave}.

At this point, we need to introduce the new wave system pertaining to the case $\k=1$.

\subsubsection{The system of acoustic-Poincar\'e waves} \label{sss:wav_Poinc-ac}

Adopting the same notation introduced in Subsection \ref{ss:k=0_wave}, we easily see that, for $\k=1$, we can recast equations \eqref{eq:reduced-MHD} as
\begin{equation} \label{eq:k=1_waves}
\left\{ \begin{array}{l}
         \veps\,\d_tr_\veps\,+\,\div V_\veps\,=\,0 \\[1ex]
         \veps\,\d_t\beta_\veps\,+\,\div W_\veps\,=\,0 \\[1ex]
         \veps\,\d_tV_\veps\,+\,\oline b\,\nabla\beta_\veps\,+\,e_3\times V_\veps\,=\,\veps\,F_\veps\,.
        \end{array}
\right.
\end{equation}
Observe that the rotation term comes into play as an $O(1)$ term, affecting the propagation of waves. Owing to the analogy with what happens for fast rotating fluids
(see \cite{C-D-G-G}, for instance),
we call the waves described by system \eqref{eq:k=1_waves} with the name of acoustic-Poincar\'e waves.

Remark that system \eqref{eq:k=1_waves} is still linear with constant coefficients, so applying the regularisation operator yields the same wave system satisfied also
by the regularised quantities, precisely as applying $\mc L_M$ allows one to pass from \eqref{eq:k=0_waves} to \eqref{eq:k=0_waves_M}.
Therefore, according with the discussion above, we can safely assume that all the functions appearing in the wave system
are smooth with respect to the space variable and use these relations when computing the limit of \eqref{eq:k=1_lim-conv},
where we also suppose that $V_\veps$ is smooth.

Next, we introduce the following  decomposition: for a three-dimensional vector-field $X$ defined over $\wtilde\Omega$, we write
\begin{equation} \label{dec:vert-av}
X(x)\,=\,\langle X\rangle(x_h)\,+\,\wtilde{X}(x)\,,\qquad\qquad \mbox{ where }\qquad
\langle X\rangle(x_h)\,:=\,\frac{1}{|\wtilde\T|}\int_{\wtilde{\mbb{T}}}X(x_h,x_3)\,dx_3\,,
\end{equation}
where $|\wtilde\T|$ denotes the Lebesgue measure of the one-dimensional torus $\wtilde\T$.
Observe that, since one simply has $|\wtilde\T|=2$, the notion of vertical average is consistent with the previous one (the one performed over $\Omega$),
when computed on functions which are independent of $x_3$, so we will not make any distinction at the notational level
(that is, we will not keep track of the constants $|\wtilde \T|$, as they will just simplify at the limit).

We remark that $\wtilde{X}$ has zero vertical average, \tsl{i.e.} $\lan \wtilde X\ran=0$. Hence we can write $\wtilde{X}(x)\,=\,\d_3\wtilde{Z}(x)$,
with $\wtilde{Z}$ having zero vertical average as well.
We also set $\wtilde{Z}\,=\,\mc{I}(\wtilde{X})\,=\,\d_3^{-1}\wtilde{X}$.

\medbreak
With the above notation, we can rewrite \eqref{eq:k=1_lim-conv} in the following form:
\begin{align}
\label{eq:k=1_split-int}
 \lim_{\veps\to0}\int^T_0\int_{\Omega}\rho_\veps\,u_\veps\otimes u_\veps:\nabla\psi\,\dx\,\dd t\,&=\,
 \lim_{\veps\to0}\frac{1}{\oline\rho}\Bigg(\int^T_0\int_{\T^2}\lan V_{\veps,h}\ran\otimes\lan V_{\veps,h}\ran :\nabla_h\nabla^\perp_h\vphi\,\dd x_h\,\dd t \\
\nonumber
&\qquad\qquad\qquad +\,\int^T_0\int_{\T^2}\lan \wtilde V_{\veps,h} \otimes\wtilde V_{\veps,h}\ran :\nabla_h\nabla^\perp_h\vphi\,\dd x_h\,\dd t\Bigg) \,.
\end{align}
We now study the convergence of each integral appearing on the right-hand side of \eqref{eq:k=1_split-int}.

\subsubsection{Convergence of the vertical averages} \label{sss:average}
We start by focusing on the first integral, namely the one involving the product of the vertical averages of the velocity fields $V_\veps$.
Since those quantities are smooth in space, similarly to what done in the proof of Lemma \ref{l:QxQ}, we can integrate 
by parts once and just focus on the term
\begin{align*}
\divh\big(\lan V_{\veps,h}\ran\otimes\lan V_{\veps,h}\ran\big)\,.
\end{align*}

As done in the case $\k=0$, simple algebraic computations allow us to write
\begin{align*}
\divh\big(\lan V_{\veps,h}\ran\otimes\lan V_{\veps,h}\ran\big)\,&=\,\divh\big(\lan V_{\veps,h}\ran\big)\;\lan V_{\veps,h}\ran\,+\,
\big(\lan V_{\veps,h}\ran\cdot\nabla_h\big)\lan V_{\veps,h}\ran \\
&=\,\divh\big(\lan V_{\veps,h}\ran\big)\;\lan V_{\veps,h}\ran\,+\,\frac{1}{2}\,\nabla_h\big|\lan V_{\veps,h}\ran\big|^2\,+\,
\lan\o_{\veps}\ran\,\lan V_{\veps,h}^\perp\ran\,,
\end{align*}
where we have defined $\o_\veps\,:=\,\curlh V_{\veps,h}\, =\, \d_1V_{\veps,2}\,-\,\d_2V_{\veps,1}$. Notice that $\lan\o_\veps\ran\,=\,\curlh\lan V_{\veps,h}\ran$.
Resorting to the notation \eqref{eq:remainder} for the remainder terms, where now the convergence must hold true for test-functions as in \eqref{eq:k=1_test}
and where we omit the dependence on $M\in\N$,
we deduce
\begin{align*}
\divh\big(\lan V_{\veps,h}\ran\otimes\lan V_{\veps,h}\ran\big)\,
&=\,\divh\big(\lan V_{\veps,h}\ran\big)\;\lan V_{\veps,h}\ran\,+\,\lan\o_{\veps}\ran\,\lan V_{\veps,h}^\perp\ran\,+\,\mc R_{\veps}\,,
\end{align*}
After computing the vertical averages of the wave system \eqref{eq:k=1_waves}, we get
\begin{align*}
 \divh\big(\lan V_{\veps,h}\ran\otimes\lan V_{\veps,h}\ran\big)\,&=\,-\,\veps\,\d_t\lan r_\veps\ran\,\lan V_{\veps,h}\ran\,+\,
 \lan\o_{\veps}\ran\,\lan V_{\veps,h}^\perp\ran\,+\,\mc R_{\veps} \\
 &=\,\lan r_\veps\ran\,\veps\,\d_t\lan V_{\veps,h}\ran\,+\, \lan\o_{\veps}\ran\,\lan V_{\veps,h}^\perp\ran\,+\,\mc R_{\veps} \\
 &=\,\veps\, \lan r_\veps\ran\,\lan F_\veps\ran\,+\,\Big(\lan\o_\veps\ran\,-\,\lan r_\veps\ran \Big)\,\lan V_{\veps,h}^\perp\ran\,-\,
\oline b\,\lan r_\veps\ran\,\nabla_h\lan\beta_\veps\ran\,+\,\mc R_\veps \\
 &=\,\Big(\lan\o_\veps\ran\,-\,\lan r_\veps\ran \Big)\,\lan V_{\veps,h}^\perp\ran\,+\,\lan \alpha_\veps\ran\,\nabla_h\lan\beta_\veps\ran\,-\,
\frac{\oline\rho}{2}\nabla_h\lan\beta_\veps\ran^2\,+\,\mc R_\veps\,,
\end{align*}
where, in the last step, we have argued similarly to the proof of Lemma \ref{l:QxQ} and invoked the scalar quantity $\alpha_\veps$. Since $\divh\psi_h=0$
in \eqref{eq:k=1_test}, the gradient term appearing in the last line is another remainder, in the sense of \eqref{eq:remainder}. Thus, we finally get
\begin{align*}
 \divh\big(\lan V_{\veps,h}\ran\otimes\lan V_{\veps,h}\ran\big)\,
 &=\,\Big(\lan\o_\veps\ran\,-\,\lan r_\veps\ran \Big)\,\lan V_{\veps,h}^\perp\ran\,+\,\lan \alpha_\veps\ran\,\nabla_h\lan\beta_\veps\ran\,+\,\mc R_\veps\,.
\end{align*}

At this point, arguing again in a similar way as in Lemma \ref{l:QxQ}, we can show that the sequence $\big(\alpha_\veps\big)_\veps$ is strongly convergent
in \tsl{e.g.} the space $L^2\big([0,T]\times\wtilde\Omega\big)$,
so one can pass to the limit $\lan \alpha_\veps\ran\,\nabla_h\lan\beta_\veps\ran\,\longrightarrow\,\lan\alpha\ran\,\nabla_h\lan\beta\ran$
in the sense of $\mc D'(\R_+\times\T^2)$, when $\veps\to0$. However, as set off by \eqref{constr:TP_k=1}, $\beta=\beta(t,x_h)$ does not depend on $x_3$ in the
case $\k=1$. Therefore $\lan\beta\ran = \beta$, which implies that
\begin{equation}
 \lan \alpha_\veps\ran\,\nabla_h\lan\beta_\veps\ran\,\longrightarrow\,\lan\alpha\ran\,\nabla_h\beta\qquad\qquad \mbox{ in the sense
 of }\ \mc D'(\R_+\times\T^2)\,,\qquad \mbox{ when }\ \veps\to0\,.
\end{equation}
Remark that the evolution $\lan\alpha\ran$ is fully determined: taking the vertical averages in \eqref{eq:alpha-lim} and using that
$u=(u_h,0)$ with $u_h=u_h(t,x_h)$ in this case, we get that this quantity satisfies the transport problem
\begin{equation} \label{eq:alpha-lim_k=1}
\d_t\lan\alpha\ran\,+\,u_h\cdot\nabla_h\lan \alpha\ran\,=\,0\,,\qquad\qquad 
\lan\alpha\ran_{|t=0}\,=\,\lan\alpha_0\ran\,:=\,\oline\rho\,\lan\beta_0\ran\,-\,\oline b\,\lan r_0\ran\,.
\end{equation}

Next, by computing the $\curlh$ of the first two components of the (vertical averages of the) momentum equation in \eqref{eq:k=1_waves},
we find
\[
\veps\, \d_t\lan\o_\veps\ran\,+\,\divh\lan V_\veps\ran\,=\,\veps\,\curlh\lan F_{\veps,h}\ran\,.
\]
Taking the difference of this relation with the averaged version of the first equation in \eqref{eq:k=1_waves} yields
\[
 \d_t \Big(\lan\o_\veps\ran\,-\,\lan r_\veps\ran\Big)\,=\,\curlh\lan F_{\veps,h}\ran\,,
\]
so the sequence $\Big(\lan\o_\veps\ran\,-\,\lan r_\veps\ran\Big)_\veps$ is compact in \tsl{e.g.} the space
$L^2\big([0,T]\times\wtilde\Omega\big)$. Such a compactness property was already discovered in \cite{F-G_RMI}
and further used, for instance, in \cite{DS-F-S-WK, F_PhysD}. Therefore, omitting the details for the sake of conciseness,
we can pass to the limit also in the term involving this quantity, to get
\[
\Big(\lan\o_\veps\ran\,-\,\lan r_\veps\ran \Big)\,\lan V_{\veps,h}^\perp\ran\,\longrightarrow\Big(\oline\rho\,\lan\o\ran\,-\,\lan r\ran \Big)\,\lan V_{h}^\perp\ran\,,
\]
where the convergence is in the sense of $\mc D'\big(\R_+\times\T^2\big)$, in the limit $\veps\to0$. Now, since the target $u$ satisfies
the constraints listed in \eqref{constr:TP_k=1} and one has the decomposition \eqref{eq:dens_dec-2} for the density functions, it is
easy to see that
\[
 \lan\o\ran = \o = \frac{\oline b}{\oline\rho}\Delta_h\beta\qquad \mbox{ and }\qquad 
\lan V_{h}^\perp\ran\,=\,\oline\rho\,\lan u_{h}^\perp\ran\,=\,-\,\oline b\,\nabla_h\beta\,.
\]
At this point, we remark that the following equalities hold true:
\[
 \Big(-\,\oline b\,\lan r\ran\,+\,\lan \alpha\ran\Big)\,\nabla_h\beta\,=\,\oline\rho\,\beta\,\nabla_h\beta\,=\,\frac{\oline\rho}{2}\,\nabla_h\beta^2\,.
\]
Now, the last term on the right identically vanishes when tested against a $\psi$ having zero horizontal divergence, as it is the case in \eqref{eq:k=1_test}.

All in all, we have proved the convergence, in the sense of distributions, of the term involving the vertical averages:
\[
  \divh\big(\lan V_{\veps,h}\ran\otimes\lan V_{\veps,h}\ran\big)\,\longrightarrow\,
  -\,{\oline b}^2\,\Delta_h\beta\,\nabla_h\beta\,.
\]
More precisely, we have shown the convergence
\begin{align}
\label{conv:average}
 \lim_{\veps\to0}\frac{1}{\oline\rho}\int^T_0\int_{\T^2}\lan V_{\veps,h}\ran\otimes\lan V_{\veps,h}\ran :\nabla_h\nabla^\perp_h\vphi\,\dd x_h\,\dd t \,&=\,
\frac{\oline b^2}{\oline\rho}\int^T_0\int_{\T^2}\Delta_h\beta\,\nabla_h\beta\cdot\nabla^\perp_h\vphi\,\dd x_h \,\dd t \\
\nonumber
&=\,-\,\frac{\oline b^2}{\oline\rho}\int^T_0\int_{\T^2}\Delta_h\beta\,\nabla_h^\perp\beta\cdot\nabla_h\vphi\,\dd x_h\,\dd t \\
\nonumber
&=\,\frac{\oline b^2}{\oline\rho}\int^T_0\int_{\T^2}\divh\left(\Delta_h\beta\,\nabla_h^\perp\beta\right)\,\vphi\,\dd x_h\,\dd t\,,
\end{align}
where, as already done in Subsection \ref{ss:k=1_first-converg}, we have adopted an abuse of notation to write the last term as an integral, instead
of a duality pairing in the sense of distributions.

\subsubsection{Convergence of the oscillatory part} \label{sss:oscillatory}

We now deal with the convergence of the second integral appearing in \eqref{eq:k=1_split-int}, namely
\begin{align}
\label{eq:int-oscill}
 \int^T_0\int_{\T^2}\lan \wtilde V_{\veps,h} \otimes\wtilde V_{\veps,h}\ran :\nabla_h\nabla^\perp_h\vphi\,\dd x_h\,\dd t\,&=\,
-\int^T_0\int_{\T^2}\divh\left(\lan\wtilde V_{\veps,h} \otimes\wtilde V_{\veps,h}\ran\right) \cdot \nabla^\perp_h\vphi\,\dd x_h\,\dd t \\
\nonumber
&=\,
-\int^T_0\int_{\T^2}\lan\divh\left(\wtilde V_{\veps,h} \otimes\wtilde V_{\veps,h}\right)\ran \cdot \nabla^\perp_h\vphi\,\dd x_h\,\dd t\,.
\end{align}

As done before, we do some algebra and compute
\begin{align*}
 \divh\left(\wtilde V_{\veps,h} \otimes\wtilde V_{\veps,h}\right)\,&=\,
\divh\big(\wtilde V_{\veps,h}\big)\;\wtilde V_{\veps,h}\,+\, \big(\wtilde V_{\veps,h}\cdot\nabla_h\big)\wtilde V_{\veps,h} \\
&=\,\divh\big(\wtilde V_{\veps,h}\big)\;\wtilde V_{\veps,h}\,+\,\frac{1}{2}\,\nabla_h\big|\wtilde V_{\veps,h}\big|^2\,+\,
\wtilde\o_{\veps}\,\wtilde V_{\veps,h}^\perp \\
&=\,\divh\big(\wtilde V_{\veps,h}\big)\;\wtilde V_{\veps,h}\,+\,\wtilde\o_{\veps}\,\wtilde V_{\veps,h}^\perp\,+\,\mc R_\veps\,,
\end{align*}
where, as above, $\o_\veps\,:=\,\curlh V_{\veps,h}\,=\,\d_1V_{\veps,2}\,-\,\d_2V_{\veps,1}$, so that $\wtilde\o_\veps\, =\,\curlh\wtilde{V}_{\veps,h}$.
Observe that, in writing the last line, we have used that $\divh\psi_h=0$ in \eqref{eq:k=1_test}.

At this point, we use a nowadays well-established structure for the oscillatory components of the solutions, already highlighted
in \cite{F-G-GV-N, F_PhysD} and which in fact goes back to work \cite{G-SR} by Gallagher and Saint-Raymond, devoted to the incompressible case.
Resorting to the notation introduced in Paragraph \ref{sss:wav_Poinc-ac}, direct computations show that we can write
\begin{align} \label{def:Theta}
\left({\curl}\wtilde{V}_{\veps}\right)_h\,&=\,\d_3\wtilde{ \Theta}_{\veps,h}\qquad\qquad\qquad \mbox{ with }\qquad
\wtilde{\Theta}_{\veps,h}\,:=\,\wtilde{V}_{\veps,h}^\perp\,-\,\d_3^{-1}\nabla^\perp_h\wtilde{V}_{\veps,3}\,.
\end{align}
Observe also that  one has $\left({\curl}\wtilde{V}_{\veps}\right)_3 = \wtilde \o_\veps$ with our notation.
By computing the $\curl$ of the momentum equation in the wave system \eqref{eq:k=1_waves}, we further see that those quantities satisfy
the following new wave system, where we get rid of the gradient term:
\begin{equation} \label{eq:wave_d_3-1}
\left\{\begin{array}{l}
        \veps\,\d_t\wtilde \Theta_{\veps,h}\,-\,\wtilde V_{\veps,h}\,=\,\veps\,\d_3^{-1}\left(\curl \wtilde F_{\veps}\right)_h \\[1ex]
        \veps\,\d_t\wtilde\o_\veps\,+\,\divh\wtilde V_{\veps,h}\,=\,\veps\,\curlh \wtilde F_{\veps,h}\,.
       \end{array}
\right.
\end{equation}

Thanks to the previous relations, and to the assumed smoothness with respect to the space variable, we can compute
\begin{align*}
\wtilde\o_{\veps}\,\wtilde V_{\veps,h}^\perp\,&=\,\veps\,\wtilde\o_\veps\,\d_t\wtilde \Theta_{\veps,h}^\perp\,-\,
\veps\,\wtilde\o_\veps\,\d_3^{-1}\left(\curl \wtilde F_{\veps}\right)_h \\
&=\,-\,\veps\,\d_t\wtilde\o_\veps\,\wtilde \Theta_{\veps,h}^\perp\,+\,\mc R_\veps \\
&=\,\divh\wtilde V_{\veps,h}\,\wtilde \Theta_{\veps,h}^\perp\,+\,\mc R_\veps\,.
\end{align*}
In light of this expression, we infer that 
\begin{align*}
 \divh\left(\wtilde V_{\veps,h} \otimes\wtilde V_{\veps,h}\right)\,
&=\,\divh\big(\wtilde V_{\veps,h}\big)\,\left(\wtilde V_{\veps,h}\,+\,\wtilde \Theta_{\veps,h}^\perp\right)\,+\,\mc R_\veps \\
&=\,\div\big(\wtilde V_{\veps}\big)\,\left(\wtilde V_{\veps,h}\,+\,\wtilde \Theta_{\veps,h}^\perp\right)\,-\,
\d_3\wtilde V_{\veps,3}\,\left(\wtilde V_{\veps,h}\,+\,\wtilde \Theta_{\veps,h}^\perp\right)\,+\,\mc R_\veps
\end{align*}
Observe that a vertical average appears in the integral to compute, see equation \eqref{eq:int-oscill}. This implies that total $\d_3$ derivatives
can be treated as remainder terms $\mc R_\veps$. Thus, we can further write
\begin{align*}
 \divh\left(\wtilde V_{\veps,h} \otimes\wtilde V_{\veps,h}\right)\,
&=\,\div\big(\wtilde V_{\veps}\big)\,\left(\wtilde V_{\veps,h}\,+\,\wtilde \Theta_{\veps,h}^\perp\right)\,+\,
\wtilde V_{\veps,3}\,\d_3\left(\wtilde V_{\veps,h}\,+\,\wtilde \Theta_{\veps,h}^\perp\right)\,+\,\mc R_\veps .
\end{align*}
Now, it follows from \eqref{def:Theta} that
\[
 \d_3\left(\wtilde V_{\veps,h}\,+\,\wtilde \Theta_{\veps,h}^\perp\right)\,=\,\nabla_h\wtilde V_{\veps,3}\,,
\]
so that the second term of the right-hand side in equality above is a pure horizontal gradient. Keeping in mind the property $\divh\psi_h=0$
which follows from \eqref{eq:k=1_test}, we thus deduce that
\begin{align}
\label{eq:expr_oscill}
 \divh\left(\wtilde V_{\veps,h} \otimes\wtilde V_{\veps,h}\right)\,
&=\,\div\big(\wtilde V_{\veps}\big)\,\left(\wtilde V_{\veps,h}\,+\,\wtilde \Theta_{\veps,h}^\perp\right)\,+\,\mc R_\veps .
\end{align}

For the remaining term in \eqref{eq:expr_oscill}, we use the first equation in \eqref{eq:k=1_waves} to get
\begin{align*}
 \div\big(\wtilde V_{\veps}\big)\,\left(\wtilde V_{\veps,h}\,+\,\wtilde \Theta_{\veps,h}^\perp\right)\,&=\,
 -\,\veps\,\d_t\wtilde r_\veps\,\left(\wtilde V_{\veps,h}\,+\,\wtilde \Theta_{\veps,h}^\perp\right) \\
 &=\,\wtilde r_\veps\,\veps\,\d_t\left(\wtilde V_{\veps,h}\,+\,\wtilde \Theta_{\veps,h}^\perp\right)\,+\,\mc R_\veps\,.
\end{align*}
At this point, we remark that a fundamental cancellation, involving the rotation terms, appears when computing the equation
for $\d_t\left(\wtilde V_{\veps,h}\,+\,\wtilde \Theta_{\veps,h}^\perp\right)$: it follows from \eqref{eq:k=1_waves} and \eqref{eq:wave_d_3-1} that
\[
 \veps\,\d_t\left(\wtilde V_{\veps,h}\,+\,\wtilde \Theta_{\veps,h}^\perp\right)\,=\,-\,\oline b\,\nabla_h\wtilde\beta_\veps\,+\,\veps\, \wtilde G_\veps\,,
\]
where $\wtilde G_\veps$ is defined in terms of the horizontal components of $\wtilde F_\veps$ and $\curl\wtilde F_\veps$.
Therefore, resorting also to the function $\alpha_\veps$ from Subsection \ref{ss:hidden} and hiding the resulting gradient term into the remainder $\mc R_\veps$,
we obtain
\begin{align*}
 \div\big(\wtilde V_{\veps}\big)\,\left(\wtilde V_{\veps,h}\,+\,\wtilde \Theta_{\veps,h}^\perp\right)\,
 &=\,-\,\oline b\,\wtilde r_\veps\,\nabla\wtilde\beta_\veps\,+\,\mc R_\veps \\
 &=\,\wtilde \alpha_\veps\,\nabla_h\wtilde\beta_\veps\,+\,\mc R_\veps\,.
\end{align*}
To conclude, we take the difference of the oscillatory components of the first two equations in \eqref{eq:k=1_waves}, one sees again that
$\big(\wtilde\alpha_\veps\big)_\veps$ is compact in \tsl{e.g.} $L^2\big([0,T]\times\wtilde\Omega\big)$,
so there is no problem in taking the limit $\veps\to0$ in the first term appearing on the right-hand side of \eqref{eq:expr_oscill}.
Since the target $\beta=\beta(t,x_h)$ does not depend on $x_3$, as established in \eqref{constr:TP_k=1},
we have that $\wtilde\beta_\veps$ is weakly convergent to $0$ for $\veps\to0$. Therefore, we deduce that
\begin{equation} \label{eq:a-b_remainder}
 \wtilde \alpha_\veps\,\nabla_h\wtilde\beta_\veps\,\longrightarrow\,0\qquad\qquad \mbox{ in the sense of }\quad \mc D'\big([0,T]\times\wtilde\Omega\big)\,,
\end{equation}
which in turn implies that $\wtilde \alpha_\veps\,\nabla_h\wtilde\beta_\veps\,=\,\mc R_\veps$ in the sense of \eqref{eq:remainder}.

Inserting this last finding into \eqref{eq:expr_oscill}, we finally find that
\begin{equation} \label{eq:conv_oscill}
 \divh\left(\wtilde V_{\veps,h} \otimes\wtilde V_{\veps,h}\right)\,=\,\mc R_\veps \qquad\qquad \mbox{ in the sense of \eqref{eq:remainder},}
\end{equation}
whenever we use a test-function $\psi$ as in \eqref{eq:k=1_test}.

\subsection{Convergence: conclusions} \label{ss:conclusions}
Plugging relations \eqref{conv:average} and \eqref{eq:conv_oscill} into \eqref{eq:k=1_split-int}, we see that, for any test-function
$\psi\,=\,\big(\nabla^\perp_h\vphi, 0\big)$ as in \eqref{eq:k=1_test}, one has the convergence
\begin{align}
\label{eq:k=1_conv_convect}
 \lim_{\veps\to0}\int^T_0\int_{\Omega}\rho_\veps\,u_\veps\otimes u_\veps:\nabla\psi\,\dx\,\dd t\,&=\,
 -\,\frac{\oline b^2}{\oline\rho}\int^T_0\int_{\T^2}\Delta_h\beta\,\nabla_h^\perp\beta\cdot\nabla_h\vphi\,\dd x_h\,\dd t \\
\nonumber
&=\,\frac{\oline b^2}{\oline\rho}\int^T_0\int_{\T^2}\divh\left(\Delta_h\beta\,\nabla_h^\perp\beta\right)\,\vphi\,\dd x_h\,\dd t\,.
\end{align}

With this relation at hand, using also the convergence properties shown in Subsections \ref{ss:k=1_first-converg} and \ref{ss:k=1_rot} and
equation \eqref{eq:alpha-lim_k=1}, we finally complete the proof of Theorem \ref{th:lim_Coriols}.



\end{document}